\documentclass[11pt, reqno]{amsart}
\usepackage{amsmath, amsthm, amscd, amsfonts, amssymb, graphicx, color, upgreek, esint}
\usepackage[bookmarksnumbered, plainpages]{hyperref}
\usepackage{tikz,float,xcolor,setspace}
\usetikzlibrary{patterns}

\usepackage[left=2.5 cm, right= 2.5 cm, top = 3 cm, bottom = 2.5 cm]{geometry}
\usepackage{parskip}

\DeclareMathOperator*{\argmin}{arg\,min}
\def\R{\mathbb{R}}
\def\eps{\varepsilon}

\def\N{\mathbb{N}}

\def\S{\mathbb{S}}
\def\lt{\left}
\def\rt{\right}

\def\cIa{\cI_\alpha}
\def\cI_2{\cI_{log}}

\def\les{\lesssim}

\def\cI{\mathcal{I}}
\def\cH{\mathcal{H}}
\def\cF{\mathcal{F}}
\def\cC{\mathcal{C}}
\def\Pb{P_{\beta}}

\newtheorem{theorem}{Theorem}[section]
\newtheorem{lemma}[theorem]{Lemma}
\newtheorem{assumption}[theorem]{Assumption}

\newtheorem{corollary}[theorem]{Corollary}
\theoremstyle{definition}
\newtheorem{definition}[theorem]{Definition}

\newtheorem{construction}[theorem]{Construction}

\theoremstyle{remark}
\newtheorem{remark}[theorem]{Remark}
\numberwithin{equation}{section}

\begin{document}
\setcounter{page}{1}

\title[Charged capillarity droplets]{Young's law for a nonlocal isoperimetric model \\ of charged capillarity droplets}

\author{Michael Goldman}
\address{CMAP, CNRS, \'Ecole Polytechnique, Institut Polytechnique de Paris, 91120 Palaiseau,
	France}
\email{michael.goldman@cnrs.fr}

\author{Matteo Novaga}
\address{Department of Mathematics, University of Pisa, 56127 Pisa, Italy}
\email{matteo.novaga@unipi.it}

\author{Adriano Prade}
\address{CMAP, \'Ecole Polytechnique, Institut Polytechnique de Paris, 91120 Palaiseau,
	France}
\email{adriano.prade@polytechnique.edu}

\begin{abstract}
    We study a variational problem modeling equilibrium configurations of charged liquid droplets resting on a surface under a convexity constraint. In the two-dimensional case with Coulomb interactions, we establish the validity of Young's law for the contact angle for small enough charges.
\end{abstract}

\maketitle


\begin{section}{Introduction}
For $n \geq 2$ let $H:=\{x \in \R^n:x_n\geq0\}$ be the upper half-space in $\R^n$ and let $Q>0$. The goal of this paper is to study existence and regularity of minimizers of the geometric variational problem
\begin{equation}\label{minproblem} \min \lt\{ \Pb(E) + Q^2 \cIa(E): E \subset H, \,E \text{ convex body}, |E|=1 \rt\}.  \end{equation}
We say that $E \subset \R^n$ is a convex body if it is convex, compact and with non-empty interior. For a convex body $E \subset H$ and $\beta \in (-1,1)$, we define
\[ \Pb(E):=\cH^{n-1}\lt(\partial E \cap \{x_n>0\}\rt) - \beta \cH^{n-1}\lt(\partial E \cap \{x_n=0\}\rt), \]
where $\cH^{n-1}$ denotes the $(n-1)$-dimensional Hausdorff measure in $\R^n$. For $\alpha \in (0,n)$, we define the $\alpha$-Riesz energy of $E$ as
\begin{equation}\label{def:Ialpha} \cI_{\alpha}(E):=\inf_{\mu(E)=1} \int_{E \times E} \frac{d\mu(x)d\mu(y)}{|x-y|^{n-\alpha}}. \end{equation}
In the limit case $\alpha=n$, the logarithmic Riesz energy of $E$ is defined as
\begin{equation}\label{defIlog}
\cI_n(E):=\inf_{\mu(E)=1} \int_{E \times E} \log \lt( \frac{1}{|x-y|}\rt) d\mu(x)d\mu(y). \end{equation}
Altogether, for $\alpha \in (0,n]$, $\beta \in (-1,1)$ and $Q>0$ we set
\begin{equation}\label{defFabQ}
 \cF_{\alpha, \beta, Q}(E):= \Pb(E) + Q^2 \mathcal{I}_\alpha(E).
 \end{equation}
The functional $\cF_{\alpha, \beta, Q}$ models the equilibrium configurations of a liquid droplet $E$ in the container $H$, resting on its surface $\partial H$, carrying a positive charge $Q$ and in absence of gravity. The Riesz energy $\cIa$ is of nonlocal nature and it accounts for the repulsive forces between the charged particles within the droplet. In the case $n=3$ and $\alpha=2$ it corresponds indeed to the Coulomb interaction energy up to multiplication by a constant, see the introduction of \cite{muratov2016well}. Instead, the capillarity term $\Pb(E)$ describes the local intermolecular forces at the interfaces between the liquid, solid and gas phases constituting the physical system, which give rise to the phenomenon of surface tension. The parameter $\beta$ is called the relative adhesion coefficient and it measures the relative strength of cohesion forces within the liquid and adhesion forces between the liquid and the container.

It is by now well-understood that local minimizers of $\Pb$ satisfy the so-called Young's law \cite{young1805iii}
\begin{equation}\label{Young} \cos \gamma = \beta,\end{equation}
where  $\gamma$ is the contact angle between $\partial E$ and $\partial H$. This compatibility condition still holds for mild perturbations of $\Pb$ including volume constraints.
We refer to \cite{finn1986equilibrium, maggi2012sets} for detailed expositions on capillarity theory and related classical results, see also the survey \cite{valdinoci2025toward} for a quick introduction. The subject has been extensively studied in recent years as well, see for example \cite{carazzato2025quantitative, chodosh2024improved, dephilippis2015regularity, de2025regularity, de2025rigidity, fusco2025isoperimetric, kreutz2024note, pacati2025some, pascale2025existence, pascale2024quantitative} and the references therein.

The question we wish to address is the validity of \eqref{Young} for minimizers of \eqref{minproblem}. This is motivated by applications in electrowetting. We refer to \cite{mugele2005electrowetting, mugele2019electrowetting} for a comprehensive guide on the subject and its applications, see also \cite{fontelos2009variational, scheid2009proof} and the references therein. The first experimental observations date back to Lippmann \cite{lippmann1875relations} in 1875. He observed that when a difference of potentials is applied to the drop, the system acts as a capacitor which leads to a flattening of the droplet with a reduction of the contact angle. He predicted the expected deviation from Young's law \eqref{Young} through what is now called  Lippmann’s law. There is however some controversy in the literature regarding the validity of this law. Indeed, it had been argued by some authors that at the microscopic level Young's law \eqref{Young} should still hold while Lippmann's law should be only apparent at the macroscopic scale. Our aim is to rigorously prove, for the simplified two-dimensional case, that \eqref{Young} indeed holds.

Before stating our main results, let us comment on the seemingly unnatural restriction we made on the class of competitors by assuming that they are convex. This is motivated by the observation from \cite{GolNovRuf13}, in the case where the container is the whole space (which corresponds to $\beta=-1$), that the variational problem is ill-posed as soon as $\alpha>1$ if we do not impose topological restrictions on the class of competitors, see also \cite{muratov2016well} as well as \cite{muratov2018equilibrium,Gol22} for the case $\alpha\le 1$. Several regularization mechanisms have been considered as a remedy, see \cite{DePHirVes19,goldman2024charged} but we follow here \cite{GolNovRuf16} where convexity was imposed. Since the observed equilibrium shapes are convex for small charges $Q$, we believe that our results are still physically relevant in this regime.

Our first result is that as in \cite{GolNovRuf16}, the class of convex sets is rigid enough to guarantee existence of minimizers.
\begin{theorem}\label{theo:existence}
    Let $n \geq 2$ and $\alpha \in (0,n]$. For all $\beta\in(-1,1)$ and $Q>0$, problem \eqref{minproblem} has a minimizer.
\end{theorem}
As a direct consequence  of the quantitative isoperimetric inequality for $\Pb$, see \cite{carazzato2025quantitative, kreutz2024note, pascale2024quantitative}, we show $L^1$ and Hausdorff convergence of minimizers of \eqref{minproblem} to isoperimetric sets for $\Pb$ as $Q \rightarrow 0$ in Lemma \ref{hausdorffconv}.

In the rest of the paper we restrict ourselves to the case  $n=\alpha=2$. Let $F \subset H$, we call $\partial F \cap \partial H$ the contact set of $F$ with $H$ or just contact set of $F$. Notice that if $E$ is a minimizer of \eqref{minproblem} then its contact set is a nonempty segment, see Remark \ref{non-empty contact}. We call contact points of $E$ the two endpoints of this segment (which could a priori coincide) and contact angles of $E$ the two angles between $\partial H$ and the tangent lines to $\partial E \cap H$ at the contact points. Notice that these are well-defined by convexity of $E$.
\begin{theorem}\label{thm:Youngintro}
    Let $E$ be a minimizer of \eqref{minproblem} and $\gamma$ be one of its contact angles, then $\partial E$ is $C^{1,1}_{loc}$ regular away from the contact points, $\mathcal{H}^1(\partial E\cap \partial H)>0$ and there exists $\bar Q >0$ such that for $Q \leq \bar Q$ we have $\cos \gamma = \beta$.
\end{theorem}
Let us discuss the main steps of the proof of Theorem \ref{thm:Youngintro}. As in  \cite{esposito2011remark, goldman2012volume, GolNovRuf16, Gol22, lamboley2023regularity}, we first relax the volume constraint into a penalization, see Lemma \ref{volconstraint}. As a consequence, every minimizer $E$ of \eqref{minproblem} satisfies a $\Lambda$-minimality property: there exists $\Lambda>0$ such that for every convex body $F \subset H$ we have
\begin{equation}\label{lambdaminintro}
    \Pb(E) + Q^2\cI_2(E) \leq \Pb(F) + Q^2 \cI_2(F) + \Lambda |E \Delta F|.
\end{equation} 
Instead of focusing on sets satisfying \eqref{lambdaminintro}, we consider a more general notion of   $\Lambda$-minimizers of capillarity shape functionals $J_\beta$ under convexity constraint in $H$, see Definition \ref{lambdamin}. The conditions we impose on the admissible  perturbations of $\Pb$ (see Assumption \ref{assloclip}) are tailored for the variations we use on the one hand and to include the logarithmic Riesz energy on the other hand. Observe however that more standard $\Lambda-$minimality properties also fit in this framework, see Remark \ref{lamboley prunier}. In Theorem \ref{C1,etareg} we first show (interior) $C^{1,\eta}_{loc}$ regularity away from the contact points for some $\eta \in (0,1]$. This follows from \cite{GolNovRuf16}.

We focus next on the validity of Young's law for $\Lambda$-minimizers of capillarity shape functionals $J_\beta$. As above, we denote by $\gamma \in (0, \pi]$ the contact angle of a $\Lambda$-minimizer $E$ with $\partial H$. We prove in  Theorem \ref{young1} that we always have the inequality
\[\cos \gamma \leq \beta.\]
The proof, which is close in spirit to \cite{GolNovRuf16} consists in a cutting procedure if the angle is too small. Notice that in order to make this construction we need to show first that $\mathcal{H}^1(\partial E\cap \partial H)>0$, which we prove in Corollary \ref{cor:poslenght}. The other direction is arguably the most interesting contribution of our work. Indeed, in Lemma \ref{examplenotyoung} we prove that in general $\Lambda$-minimizers of $\Pb$ do not satisfy $\cos \gamma \geq \beta$. We show however that it holds provided $\partial E \cap \{x_2 >0\}$ does not locally coincide with its tangent cone, see Theorem \ref{young2}. Indeed, thanks to this hypothesis, we may construct exterior perturbations which both preserve convexity and reduce the contact angle.
\begin{remark}\label{highern}
It is tempting  to try to extend the results from  \cite{lamboley2023regularity},  to the setting of capillarity shape functionals. There are however some severe difficulties. The first one is apparent from Lemma \ref{examplenotyoung} which shows that in general Young's law cannot hold. The second one is that as opposed to \cite{lamboley2023regularity}, which is about estimates, we are interested here in an equation which would thus require to be much more precise. Finally, even for the interior problem, it is unclear whether the results of \cite{lamboley2023regularity} can cover the case of the capacity, see however \cite{prunier2024fuglede}. \end{remark}

In order to apply these general results for $\Lambda-$minimizers of $J_\beta$ to the case $J_\beta=\cF_{2, \beta, Q}$, we prove in Theorem \ref{strictconv} that for small charges,  the boundary of a minimizer for \eqref{minproblem} contains no segments in the interior of $H$.

We conclude the introduction with an open problem. From Theorem \ref{strictconv} it is natural to expect that for small charges, minimizers of \eqref{minproblem} should be uniformly convex in $\{x_2>0\}$. This would allow to write the  Euler-Lagrange equation,
\begin{equation}\label{eulerlagintro} \kappa_E - \frac{Q^2}{2\pi} |\nabla u_E|^2 = \lambda \qquad \text{on } \partial E \cap H. \end{equation}
Here $\kappa_E$ and $u_E$ denote respectively the curvature and the logarithmic potential of a minimizer $E$, see Section \ref{potentialprel}. Using a simple bootstrap argument we would then obtain smoothness of $\partial E$ away from the contact points. In Lemma \ref{positivecurvature}, we obtain a first step in this direction.

The paper is organized as follows. In Section \ref{prem} we set the notation and present some preliminary results on the Riesz and logarithmic energies, harmonic measure in dimension $2$ and capillarity theory. In Section \ref{sec4} we prove existence of minimizers for \eqref{minproblem}. In Section \ref{sec5} we introduce the notion of $\Lambda$-minimizers of capillarity shape functionals and study their properties. Finally, in Section \ref{sec3} we show Young's law and regularity of minimizers of \eqref{minproblem}.

\end{section}

\begin{section}{Preliminaries}\label{prem}

We work in $\R^n$ with $n \geq 2$, we write $x = (x_1, ... , x_n)$ for $x \in \R^n$ and $\{e_i\}_{i=1}^n$ to denote the standard orthonormal basis. If $E\subset \R^n$ then we denote by $\chi_E$ its indicator function and by $E^\circ$ its topological interior. If $E$ is measurable then $|E|$ denotes its Lebesgue measure and $\cH^k(E)$ denotes its $k$-dimensional Hausdorff measure for $k\in [0,n]$. For $E \subset\R^n$ convex body and $\Omega \subset \R^n$ measurable we write $P(E,\Omega):=\cH^{n-1}(\partial E \cap \Omega)$ to denote the perimeter of $E$ in $\Omega$. If $\Omega=\R^n$ then we drop the dependence on $\Omega$ and we just write $P(E)$. We denote the diameter of a convex body $E \subset \R^n$ by $\text{diam}(E)$. For $x \in \partial E$, we denote the tangent cone of $E$ at $x$ by
\[ C_{x}:=\bigcup_{\lambda >0} \lambda\lt(E - x\rt).\]
We denote by $\nu_E$ the outward normal to $E$.
For $x\in \R^n$, $r>0$ we set $B_r(x):=\{y\in \R^n: |x-y|<r\}$, we drop the dependence on $x$ if $x=0$ and we also write $\omega_n:=|B_1|$. If $x,y \in \R^n$, we write $[x,y]$ for the segment joining them.
We often use constants $C,c>0$ for estimates. Their dependence is given in the corresponding statement or proof and their value may change form line to line. We write $A \lesssim B$ if there exists $C>0$ depending on $n, \alpha,\beta$ such that $A \leq C\,B$ and $A \simeq B$ if $A \lesssim B$ and $B \lesssim A$. We also denote $A \ll B$ to indicate that there is a  constant $\eps>0$ depending on $n, \alpha,\beta$ such that $A \leq \eps B$. Finally, for $k \in \N \cup \{0\}$ and $\eta\in(0,1]$ then $C^{k,\eta}(\Omega)$ denotes the H\"older space of $k$ times differentiable functions with $k^{th}$ derivative $\eta$-H\"older.


\begin{subsection}{Riesz energies, potentials and equilibrium measures}\label{potentialprel} $\,$

We recall the definition and various properties of the Riesz energy $\cIa$ and the logarithmic energy $\cI_n$, see \cite{GolNovRuf13, GolNovRuf16}, the notes \cite{prats2023notes} and the monographs \cite{garnett2005harmonic, doohovskoy_foundations_1972, saff2013logarithmic} for more details on this subject. For $\alpha\in(0,n]$ we set
\[ k_\alpha(z)=\begin{cases}
    |z|^{-(n-\alpha)} &\textrm{if } \alpha \in (0,n)\\
    -\log |z| &\textrm{if } \alpha=n
    \end{cases} \]
For a positive measure $\mu$ we set
\[ I_\alpha(\mu)= \int_{\R^n\times\R^n} k(x-y)d\mu(x)d\mu(y) \]
so that for a compact set $E\subset \R^n$, recall \eqref{def:Ialpha} and \eqref{defIlog},
\[ \cIa(E)=\min_{\mu(E)=1} I_\alpha(\mu) \]
By definition of $\cIa$ we have the monotonicity property, $ E \subset F$ implies $\cIa(E) \geq \cIa(F)$.
In addition, from a change of variable we get that for every $\lambda>0$
\begin{equation}\label{scaling} \cIa(\lambda E) = \lambda^{-(n-\alpha)} \cIa(E) \qquad \textrm{if } \alpha \in (0,n) \qquad \textrm{and} \qquad  \cI_n(\lambda E) = \cI_n(E) - \log(\lambda).
\end{equation}
From \cite[p. 131]{doohovskoy_foundations_1972} we know that if $E \subset \R^n$ is compact with $\cIa(E) < +\infty$ then the infimum in the definition of $\cIa(E)$ is achieved by a unique measure $\mu_E$. We have that $\text{spt}(\mu_E)= E$ for $\alpha \in (0,2)$, whereas spt$(\mu_E) \subset \partial E$ for $n \geq 3$, $\alpha \in [2,n)$ and for $n=\alpha=2$. We then define the potential
\[ u_E:= k_\alpha \ast \mu_E,\]
so that in particular $\cIa(E)=\int_{E} u_E d\mu_E$. Let us notice that $\cI_{n}(E)$ may be negative, but $\cI_n(E)>0$ if diam$(E)<1$.

We extend now the semicontinuity and continuity properties of $\cIa$  with respect to Hausdorff convergence of compact sets from \cite[Lemma 3.4]{goldman2024charged} and \cite[Theorem 4.2]{GolNovRuf13} to the case $\alpha=n$. Let us point out that the argument from \cite[Theorem 4.2]{GolNovRuf13} does not extend directly to this case and that the proof we provide here is somewhat simpler. We say that $E \subset \R^n$ has interior density bounds if there exist $C>0$ and $r_0>0$ such that for all $x \in E$ we have
\[ |E \cap B_r(x)| \geq Cr^n \qquad \text{for all } 0<r \leq r_0.\]
\begin{lemma}\label{contIa}
    Let $\{E_k\}_{k\in \N} \subset \R^n$ be sequence of compact sets converging in the Hausdorff topology to a compact set $E$, then for all $\alpha \in (0,n]$ we have $\cIa(E) \leq \liminf_{k \rightarrow +\infty} \cIa(E_k)$. If in addition $E$ has interior density bounds then $\cIa(E) = \lim_{k \rightarrow +\infty} \cIa(E_k)$.
\end{lemma}
\begin{proof}
    The case $\alpha<n$ is treated in \cite[Lemma 3.4]{goldman2024charged} and \cite[Theorem 4.2]{GolNovRuf13} so we only consider the case $\alpha=n$. The proof is analogous to the case $\alpha<n$ provided we can prove the following claim. Let $E \subset \R^n$ be a compact set with interior density bounds, then we have
    \[ \cI_n(E)= \inf \lt \{ I_{n}(\mu) : \mu=f\,dx ,\, f \in L^{\infty}(E), \, \int_E f \,dx=1\rt \}. \]
    By dilation of $E$ we may assume that $E \subset B_\delta$ for some $\delta \ll 1$, so that for every $x,y \in E$ we have $|x-y|\ll 1$. Let $\mu=\mu_E$ be the equilibrium probability measure for $\cI_n(E)$ and $u=u_E$ be its associated potential function. For $\eps > 0$ we define the measure $\mu_\eps:=f_\eps\,\cH^n \llcorner  E$ with
    \[ f_\eps(z):= \int_{E\cap B_\eps(z)} \frac{d\mu(y)}{|E \cap B_\eps(y)|}. \]
    By definition spt$(\mu_\eps) \subset E$, moreover by Fubini theorem and the fact that for all $y,z \in \R^n$ there holds $\chi_{B_\eps(y)}(z) = \chi_{B_\eps(z)}(y)$ we have
    \[ \begin{split} \mu_\eps(E) & = \int_E f_\eps(z)dz= \int_E \int_E \frac{\chi_{B_\eps(z)}(y)}{|B_{\eps}(y)\cap E|}d\mu(y)\,dz = \int_E \int_E \frac{\chi_{B_\eps(y)}(z)}{|B_{\eps}(y)\cap E|}\,dz\, d\mu(y) \\ &  = \int_E d\mu(y) = 1. \end{split}  \]
    We also have $f_\eps \in L^{\infty}(E)$ with $\|f_\eps \|_{L^{\infty}(E)} \lesssim \eps^{-n}$ thanks to the interior density bounds, indeed if $\eps$ is small enough
    \[ f_{\eps}(x) \leq \lt( \min_{y \in E} \lt|E \cap B_\eps(y)\rt|\rt)^{-1} \leq C \eps^{-n}.\]
    We denote by $u_\eps=k_n\ast \mu_\eps$ the logarithmic potential function of $\mu_\eps$. The goal is to show that $I_{n}(\mu_\eps) \rightarrow I_{n}(\mu)$ as $\eps \rightarrow 0$, that is
    \begin{equation}\label{weakstrong}  \lim_{\eps\to 0} \int_{E} u_\eps d\mu_\eps= \int_E u d\mu. \end{equation}
    We begin with a preliminary observation. Using Fubini as above, we see that for every measurable function $\varphi: \R^n \rightarrow \R$ there holds
    \begin{equation}\label{averageint}
         \int_E \varphi(z)\mu_\eps(z)\,dz = \int_E \lt( \frac{1}{|E \cap B_\eps(y)|} \int_{E \cap B_\eps(y)} \varphi(z)dz \rt) d\mu(y).
    \end{equation}
    Using this a first time we may write
    \[ \int_{E} u_\eps d\mu_\eps=\int_E v_\eps d\mu \]
    where
    \[ v_\eps(x):= \frac{1}{|E\cap B_\eps(x)|}\int_{E\cap B_\eps(x)} u_\eps(y) dy. \]
    Our aim is to show that
    \begin{equation}\label{upperrough}
     v_\eps\les u
    \end{equation}
    and that for $\mu$-a.e. $x \in \R^2$,
    \begin{equation}\label{pointwiseconv}
         \lim_{\eps\to 0} v_\eps(x)=u(x).
    \end{equation}
    Using dominated convergence this would conclude the proof of \eqref{weakstrong}. To prove these two claims we write that by definition
    \[ u_\eps(y)=\int_E \log\lt(\frac{1}{|y-z|}\rt) d\mu_\eps(z) \]
    so that applying \eqref{averageint} once more we get
    \[ u_\eps(y)= \int_{E} \lt(\frac{1}{|E\cap B_\eps(w)|}\int_{E\cap B_\eps(w)} \log\lt(\frac{1}{|y-z|}\rt) dz\rt) d\mu(w). \]
    Using Fubini this yields    
    \[ v_\eps(x)=\int_{E} \xi_\eps(w,x)\,d\mu(w), \]
    where
    \[ \xi_\eps(w,x)= \frac{1}{|E\cap B_\eps(w)||E\cap B_\eps(x)|}\int_{(E\cap B_\eps(x))\times(E\cap B_\eps(w))} \log\lt(\frac{1}{|y-z|}\rt) dz dy. \]
    We claim that for $w\neq x$
    \begin{equation}\label{secondclaimrough}
        \xi_\eps(w,x)\les \log\lt(\frac{1}{|w-x|}\rt)
    \end{equation}
    and that there holds
    \begin{equation}\label{secondclaimpointwise}
    \lim_{\eps\to 0 } \xi_\eps(w,x)=\log\lt(\frac{1}{|w-x|}\rt).
    \end{equation}
    Indeed, \eqref{secondclaimrough} implies
    \[ v_\eps(x)\les \int_E \log\lt(\frac{1}{|w-x|}\rt) d\mu(w)=u(x) \]
    and thus \eqref{upperrough}, whereas combining \eqref{secondclaimrough} and \eqref{secondclaimpointwise} yields by dominated convergence
    \[ \lim_{\eps\to 0} v_\eps(x)=\lim_{\eps\to 0} \int_{E} \xi_\eps(w,x) d\mu(w)=\int_{E} \log\lt(\frac{1}{|w-x|}\rt) d\mu(w)=u(x) \]
    at all points where $u(x)<\infty$ (which by definition holds $\mu$ a.e.). This in turn concludes the proof of \eqref{pointwiseconv}. We first point out that \eqref{secondclaimpointwise} is immediate if $w\neq x$ by continuity of log on $\R^+$ so we only need to show \eqref{secondclaimrough}. Thanks to the interior density bounds, we have
    \[ \xi_\eps(w,x)\les \eps^{-2n} \int_{ B_\eps(x)\times B_\eps(w)} \log\lt(\frac{1}{|y-z|}\rt) dz dy. \]
    By translation invariance of the right-hand side, we may assume that $w=0$. We are  thus left with the proof of
    \[ \eps^{-2n} \int_{ B_\eps(x)\times B_\eps} \log\lt(\frac{1}{|y-z|}\rt) dz dy\les \log\lt(\frac{1}{|x|}\rt). \]
    If $|x|\gg \eps$ this estimate is obvious, hence it is enough to prove that for $|x|\les \eps$,
    \[ \eps^{-2n} \int_{ B_\eps(x)\times B_\eps} \log\lt(\frac{1}{|y-z|}\rt) dz dy\les \log\lt(\frac{1}{\eps}\rt). \]
    By Riesz rearrangement inequality \cite[Theorem 3.7]{lieb2001analysis} the left-hand side is maximal for $x=0$ and by scaling we have
    \[ \eps^{-2n} \int_{ B_\eps\times B_\eps} \log\lt(\frac{1}{|y-z|}\rt) dz dy=|B_1|^2 \log\lt(\frac{1}{\eps}\rt)+ \int_{ B_1\times B_1} \log\lt(\frac{1}{|y-z|}\rt) dz dy\les \log\lt(\frac{1}{\eps}\rt). \]
    This concludes the proof. \end{proof}
\end{subsection}

\begin{subsection}{Harmonic and equilibrium measure in dimension 2}\label{harmonic measure}$\,$

We now study more in details the logarithmic equilibrium measure in dimension $2$,  and connect it with the harmonic measure. Let $E \subset \R^2$ be a compact and simply connected Lipschitz set with non-empty interior. The potential function $u_E$ satisfies
\begin{equation}\label{logequilibrium} \begin{cases}
    -\Delta u_E = 2\pi \mu_E & \text{in } \mathcal{D}'(\R^2) \\ u_E = \cI_2(E) & \text{on } E \\ u_E < \cI_2(E) & \text{in } \R^2 \setminus E,
\end{cases}\end{equation}
see \cite[Chapter 2, Theorem 1.3 and pp. 138,139]{saff2013logarithmic}, \cite[Section 2.4]{doohovskoy_foundations_1972} and \cite[Chapter 3]{garnett2005harmonic}. Assuming that $0 \in E^\circ$, the Green function $G^{\infty}_{E^c}$ with pole at infinity is the solution to
\begin{equation}\label{Greeninfi} \begin{cases}
    - \Delta u =0 & \text{in } E^c \\ u>0 & \text{in } E^c \\ u=0 & \text{on } \partial E \\ \lim_{|z| \rightarrow +\infty} u(z) - \frac{1}{2\pi} \log |z| = \frac{1}{2\pi}\cI_2(E).
\end{cases}\end{equation}
The harmonic measure $\omega^{\infty}_{E^c}$ of $E^c$ with pole at infinity is the unique probability measure on $\partial E$ satisfying
\begin{equation}\label{harminfi} \int_{\partial E} \varphi d\omega^{\infty}_{E^c} = \int_{E^c} G^{\infty}_{E^c} \Delta \varphi \quad \text{for all } \varphi \in C^{\infty}_c(\R^2) \quad \Longleftrightarrow \quad -\Delta G^{\infty}_{E^c} = - \omega^{\infty}_{E^c} \quad \text{in } \mathcal{D}'(\R^2).\end{equation}
See \cite[Theorem 7.32]{prats2023notes} for more details on the construction of $G^{\infty}_{E^c}$ and $\omega^{\infty}_{E^c}$. From now on we drop the dependence of $G^{\infty}$ and $\omega^{\infty}$ on the domain $\Omega$ whenever it is clear from the context. Let $\mu_E$ be the unique minimizing measure for $\cI_2(E)$, so that  spt$(\mu_E)=\partial E$ and $u_E$ is a solution to \eqref{logequilibrium}. If we let
\[ u:= \frac{1}{2\pi} \lt(\cI_2(E) - u_E \rt),\]
we see that it satisfies \eqref{Greeninfi}, hence $u$ is the Green function with pole at infinity of $E^c$. By \eqref{logequilibrium} we have that $-\Delta u = - \mu_E$, so from \eqref{harminfi} we deduce that $\mu_E$ coincides with the harmonic measure of $E^c$ with pole at infinity. We now recall a classical result that will be used in the following.

\begin{lemma}\label{abscontmeas}
    Let $E \subset \R^2$ be compact, simply connected and of class $C^1$, if $u_E \in C^1(\R^2 \setminus E^\circ)$ then $\mu_E$ is absolutely continuous with respect to $\cH^1 \llcorner \partial E$ and writes
\begin{equation}\label{regeqmeas} \mu_E = \frac{|\nabla u_E|}{2\pi} \cH^1 \llcorner \partial E. \end{equation} 
\end{lemma}
The proof follows by adapting the argument in \cite[Proposition 2.22]{GolNovRuf13} to dimension $n=2$. Still, if $E$ is a Lipschitz domain then $\omega^{\infty}= \mu_E$ is absolutely continuous with respect to $\cH^1 \llcorner \partial E$ and there exists $p>1$ such that the density of $\omega^{\infty}$ is in $L^p(\partial E)$, see \cite[Chapter 7, Theorem 4.2]{garnett2005harmonic}. For convex bodies it is possible to upgrade the previous result to $p>2$ as stated in the following result, which sums up \cite[Lemma 3.5 and Theorem 3.1]{GolNovRuf16}. We recall that given a convex body $E \subset \R^2$ with $x \in \partial E$, the angle of $E$ at $x$ is the angle $\gamma \in (0, \pi]$ spanned by the tangent cone $\bigcup_{\lambda > 0} \lambda \lt( E - x\rt)$.

\begin{theorem}\label{convexint}
    Let $E \subset \R^2$ be a convex body with minimal angle $\gamma \in (0, \pi]$, $\{E_k\}_{k\in \N}$ be a sequence of convex bodies converging to $E$ in the Hausdorff topology and $\mu_E$, $\{\mu_{E_k}\}_{k\in \N}$ be the associated equilibrium measures with densities $f_E$, $\{f_{E_k}\}_{k\in \N}$. Then for all $1 \leq p < (2\pi - \gamma)/(\pi -\gamma)$ there exists $C_{p,E}$ such that for $k$ large enough depending on $p$ we have
    \[ f_{E_k} \in L^{p}(\partial E_k) \qquad \text{with} \qquad \|f_{E_k}\| \in L^p(\partial E_k) \leq C_{p,E}.\]
    In particular, the equilibrium measure of any convex body is in $L^p(\partial E)$ for some $p>2$.
\end{theorem}

\begin{remark}\label{non-bounded eq}
    If the boundary of a convex body $E$ is non-regular at a point $x$ (so the corresponding angle of $E$ at $x$ is not equal to $\pi$) then the density of the optimal measure $\mu_E$ is not bounded at $x$. This is highlighted in \cite[Example 6.5]{GolNovRuf13}, where the authors considered a harmonic function defined in the exterior of a convex cone and with constant boundary conditions. The same example also provides a motivation for the integrability threshold $(2\pi - \gamma)/(\pi - \gamma)$ appearing in the previous theorem.
\end{remark}
\end{subsection}

\begin{subsection}{Capillarity}\label{capillarity} $\,$

In this section we recall some classical results about capillarity theory in the upper half-space $H=\{ x \in \R^n:x_n \geq 0\}$, see \cite[Chapter 19]{maggi2012sets}. Let us point out that all the results described here hold also in the larger class of sets of finite perimeter but for simplicity we reduce ourselves to convex sets. For $\beta \in (-1,1)$, if $E \subset H$ is a convex body we denote
\[ \Pb(E) = P\lt(E,H^\circ\rt) - \beta P\lt(E,\partial H\rt).\]
Notice that  $P(E)=P\lt(E, H^\circ\rt) +P\lt(E,\partial H\rt)$. By \cite[Proposition 19.22]{maggi2012sets} we also have
\begin{equation}\label{cakeestimate}
   P\lt(E,H^\circ \rt) > P\lt(E,\partial H\rt).
\end{equation}
From \eqref{cakeestimate}, we infer that for $\beta \in (-1,1)$
\[ \Pb(E) = \frac{1+\beta}{2} \lt(P\lt(E, H^\circ \rt) - P\lt(E,\partial H\rt) \rt) + \frac{1-\beta}{2} P(E) \geq \frac{1-\beta}{2} P(E). \]
In particular, for any $\beta \in (-1, 1)$ and $E \subset H$ we have 
\begin{equation}\label{capestimate2}
    P(E) \leq C_{\beta} \, \Pb(E).
\end{equation}
We now consider the variational problem
\begin{equation}\label{minproblemcap}
    \min \lt \{ \Pb(E) : E \subset H, |E|=1 \rt\}.
\end{equation}
By \cite[Proposition 19.21]{maggi2012sets}, minimizers $E^{\beta}$ of \eqref{minproblemcap} are suitably truncated balls lying on $\partial H$. More precisely, setting $T^{\beta}:=\{x \in \overline{B_1}: x\cdot e_n \geq \beta\}$ and
\begin{equation}\label{minPbeta} B^{\beta}=B^{\beta}(1) := \frac{1}{|T^{\beta}|^{1/n}}\lt(T^{\beta} - \beta e_n \rt),\end{equation}
minimizers are sets of the form $B^{\beta}_x:=B^{\beta} + x$ for $x\in \partial H$. If instead we consider the problem \eqref{minproblemcap} with volume constraint $|E|=m$ for some $m >0$, the same holds with sets of the form
\[ B^{\beta}(m) := \frac{m^{1/n}}{|T^{\beta}|^{1/n}}\lt( T^{\beta} - \beta e_n \rt) \]
and their translations $B^{\beta}_x(m):=B^{\beta}(m) + x$ for $x\in \partial H$. Let now $\gamma_\beta$ be the contact angle between $B^{\beta}$ and $\partial H$. By construction we have the validity of Young's law
\[ \cos \gamma_{\beta} = \beta.\]
Denoting by $\nu_H$ and $\nu^{\beta}$ the outer unit normals of $H$ and $B^{\beta}$ respectively, the equivalent formulation of Young's law $\nu_H \cdot \nu^{\beta} = - \beta$ holds at the contact points between $\partial B^{\beta} \cap H$ and $\partial H$. As a consequence of minimality we have the following isoperimetric inequality, if $E \subset H$ is a convex body then
\[ \frac{\Pb(E)}{|E|^{\frac{n-1}{n}}} \geq \frac{\Pb(B^{\beta})}{|B^{\beta}|^{\frac{n-1}{n}}} \]
with equality if and only if $E=B^{\beta}(|E|)$ up to a translation. Let $E \subset H$ be a convex body such that $|E|=m$, we define the Fraenkel asymmetry and the isoperimetric deficit respectively as
\begin{equation}\label{fraenkel} \alpha_{\beta}(E) := \inf \lt\{ \frac{|E \Delta B^{\beta}_x(m)|}{|B^{\beta}_x(m)|}: x \in \partial H \rt\} \qquad \text{and} \qquad D_{\beta}(E):= \frac{\Pb(E)-\Pb(B^{\beta}(m))}{\Pb(B^{\beta}(m))}. \end{equation}
The Fraenkel asymmetry  measures the $L^1$ distance between the competitor $E$ and the set of minimizers of $\Pb$. We finally recall the sharp quantitative isoperimetric inequality for $\Pb$, see \cite{kreutz2024note,pascale2024quantitative,carazzato2025quantitative}. For any $E \subset H$,
\begin{equation}\label{quantisop}
    \alpha_{\beta}(E)^2 \les D_{\beta}(E).
\end{equation}

\end{subsection}
\end{section}

\begin{section}{Existence of minimizers}\label{sec4}

In this section we prove Theorem \ref{theo:existence} as well as convergence of minimizers of \eqref{minproblem} to minimizers of $\Pb$ as the charge vanishes. Finally, when $n=\alpha=2$ we show how that it is possible to relax the volume constraint $|E|=1$.
We recall the definition \eqref{defFabQ} of $\cF_{\alpha, \beta, Q}$.
If either $\beta \geq 1$ or $\beta \leq -1$, then we can respectively prove non-existence of minimizers or reduce the problem to the study of the functional $P+Q^2\cIa$. This is relatively standard and is discussed in the following remarks.

\begin{remark}[Non-existence of minimizers when $\beta \geq 1$]\label{nonexistence} To show non-existence we follow the same strategy as \cite[Remark 19.19]{maggi2012sets} and for $\alpha \in (0,n]$ we consider
\begin{equation}\label{illposedinf}
    \psi(\alpha, \beta) := \inf \lt\{ \cF_{\alpha, \beta, Q}(E) : E \subset H,\, E \text{ convex body},|E|=1\rt\}.
\end{equation}
For $R>0$, we set $\eps_R:= \lt(\omega_{n-1}R^{n-1} \rt)^{-1}$ and build a sequence of convex bodies $\{E_R\}_{R>0} \subset H$ defined as
\[ E_R := \lt\{ x\in \R^n: \lt| x_1^2 + \ldots + x_{n-1}^2\rt|^{1/2}\le R, 0\le x_n \le \eps_R\rt\}.\]
By construction we have $|E_R|=1$ for all $R>0$. We claim that for all $\alpha \in (0,n)$ we have $\cIa(E_R) \rightarrow 0$  and $\cI_n(E)\to -\infty$ as $R \rightarrow + \infty$. Leaving the proof of the claim for later, we compute
\[ \begin{split} \cF_{\alpha, \beta, Q}(E_R) & =(1-\beta)\omega_{n-1}R^{n-1} + (n-1)\omega_{n-1}R^{n-2}\eps_R + Q^2 \cIa(E_R) \\ & = (1-\beta)\omega_{n-1}R^{n-1} + \frac{n-1}{R} + Q^2 \cIa(E_R). \end{split}\]
If $\beta>1$ or $\beta=1$ and $\alpha=n$, then $\cF_{\alpha, \beta, Q}(E_R) \rightarrow - \infty$ as $R \rightarrow + \infty$, so $\psi(\alpha, \beta)=-\infty$ and no minimizer exists. If instead $\beta=1$ and $\alpha\in(0,n)$; then $\cF_{\alpha, 1, Q}(E_R) \rightarrow 0$ as $R \rightarrow + \infty$, so that $\psi(\alpha,1) = 0$. However by \eqref{cakeestimate} we get that $P_1(E)>0$ for all $E \subset H$, hence minimizers cannot exist either. It remains to prove the claim above. Since the case $\alpha=n$ is easier we only consider the case $\alpha<n$. We denote by $D_R$ the $(n-1)$-dimensional disk in $\R^n$ and set
\[ E'_R:= D_R \times (-\eps_R/2, \eps_R/2). \]
By translation invariance of $\cIa$ we have $ \cIa(E_R) = \cIa(E'_R)$, so it is equivalent to prove that $\cIa(E'_R) \rightarrow 0$ as $R \rightarrow + \infty$. We start with the easier case $\alpha \in (1,n)$. Given $M>0$ we have $\cIa(E'_R)\leq \cIa(D_M)$ for all $R >M$ by monotonicity of $\cIa$ under inclusion, passing to the limit as $R \rightarrow + \infty$ we obtain
\begin{equation}\label{Ialphalim} \lim_{R\rightarrow +\infty} \cIa(E'_R) \leq \cIa(D_M) \stackrel{\eqref{scaling}}{=} \frac{1}{M^{n-\alpha}} \cIa(D_1) \qquad \text{for all } M>0.\end{equation}
To estimate $\cIa(D_1)$ we test its definition with the uniform measure $\omega_{n-1}^{-1} \cH^{n-1} \llcorner D_1$ on $D_1$ and, denoting $x=(x', x_n)$ with $x'\in \R^{n-1}$, we get
\[ \begin{split} \cIa(D_1) &  \les \int_{D_1} \int_{D_1} \frac{dy'}{|x'-y'|^{n-\alpha}}dx' \lesssim \int_{D_1} \int_{D_1} \frac{dy'}{|y'|^{n-\alpha}}dx' \\ & \lesssim \int_{D_1} \frac{dy'}{|y'|^{n-\alpha}} \lesssim \int_0^1 \frac{d\rho}{\rho^{2-\alpha}}. \end{split}\]
The final quantity is finite as $\alpha \in (1,n)$, hence we pass to the limit as $M \rightarrow +\infty$ in \eqref{Ialphalim} and prove the claim. To study the remaining case $\alpha \in (0,1]$, we rescale $E'_R$ to $F_R:= D_1 \times (-h_R/2, h_R/2)$ with $h_R= (\omega_{n-1}R^n)^{-1}$, so that $|F_R|= R^{-n}$ and $E'_R=RF_R$. Again, we exploit the scaling property  \eqref{scaling} of $\cIa$, test the definition of $\cIa(F_R)$ with the uniform measure on $F_R$ and use that for $x\in F_R$ we have $F_R\subset 2 F_R +x$ to find
\[ \begin{split} \cIa(E'_R) & = \frac{1}{R^{n-\alpha}}\cIa(F_R) \leq \frac{1}{R^{n-\alpha}} \frac{1}{|F_R|^2} \int_{F_R}\int_{F_R} \frac{dy}{|x-y|^{n-\alpha}}dx \\ & = R^{n+\alpha} \int_{F_R} \int_{F_R} \frac{dy}{|x-y|^{n-\alpha}} dx \leq  R^{n+\alpha} \int_{F_R} \int_{2F_R+x} \frac{dy}{|x-y|^{n-\alpha}} dx\\
&\leq R^{\alpha} \int_{2F_R} \frac{dy}{|y|^{n-\alpha}} \les R^{\alpha} \int_{F_R} \frac{dy}{|y|^{n-\alpha}}. \end{split} \]
To estimate the last quantity we cover the integration domain with $F_R \setminus \lt( D_{h_R} \times (-h_{R/2}, h_{R/2})\rt)$ and $B_{2h_{R}}$. We first consider the case $\alpha \in (0,1)$, we compute
\[ \int_{B_{2h_R}} \frac{dy}{|y|^{n-\alpha}} \lesssim \int _0^{2h_R} \frac{d\rho}{\rho^{1-\alpha}} = \frac{(2h_R)^{\alpha}}{\alpha} \lesssim \frac{1}{R^{n\alpha}}\]
and
\begin{equation}\label{secondcomp} \int_{F_R \setminus \lt( D_{h_R} \times (-h_{R/2}, h_{R/2})\rt)} \frac{dy}{|y|^{n-\alpha}} \leq h_R \int_{D_1 \setminus D_{h_{R}}} \frac{dy'}{|y'|^{n-\alpha}} \lesssim h_{R} \int_{h_R}^1 \frac{d\rho}{\rho^{2-\alpha}} \lesssim h_R^{\alpha} \lesssim \frac{1}{R^{n\alpha}}. \end{equation}
Altogether we find
\[ \cIa(E'_R) \lesssim R^{-(n-1)\alpha} \]
and we obtain the claim by passing to the limit as $R \rightarrow \infty$. If instead $\alpha =1$ the same argument applies, except that the last two terms in \eqref{secondcomp} are replaced by
\[-h_R \log h_R \lesssim \frac{\log R}{R^n}. \]
Therefore, when $\alpha = 1$ we have
\[ \cIa(E'_R) \lesssim \frac{1}{R^{n-1}}+ \frac{\log R}{R^{n-1}} \]
and as before we conclude $\cIa(E'_R) \rightarrow 0$ as $R \rightarrow + \infty$. \end{remark}

\begin{remark}[No capillarity phenomena when $\beta \leq -1$] If $\beta \leq -1$ then for $\alpha\in(0,n]$ and $E \subset H$ we have that
\[ \cF_{\alpha, \beta, Q}(E) = P(E) + (-\beta -1)P(E, \partial H) + Q^2 \cIa(E) \geq P(E) + Q^2 \cIa(E). \]
Hence the study of \eqref{illposedinf} reduces to that of $\cF_{\alpha, Q}(E)= P(E) + Q^2 \cIa(E)$ in $\R^n$ and under convexity constraint, already addressed in \cite{GolNovRuf16}.
\end{remark}
We now turn to the existence of minimizers of \eqref{minproblem}. We recall that, by \cite[Lemma 4.1]{esposito2005quantitative}, for every $n\geq 2$, if $E \subset \R^n$ is a convex body then we have
\begin{equation}\label{diamestimate}
    \text{diam}(E) \les \frac{P(E)^{n-1}}{|E|^{n-2}}.
\end{equation}

\begin{proof}[Proof of Theorem \ref{theo:existence}]
    Let $\{E_k\}_{k\in\N} \subset H$ be a minimizing sequence for \eqref{minproblem}, namely
    \[ \lim_{k \rightarrow +\infty} \cF_{\alpha, \beta, Q}(E_k) = \inf \left \{ \cF_{\alpha, \beta, Q}(E) : E \subset H,\, E \text{ convex body},\, |E|=1 \right \} < +\infty. \]
    We show that the diameters $\{\text{diam}(E_k)\}_{k \in \N}$ are uniformly bounded. In the case $\alpha \in (0,n)$ we have $\cIa(E_k) \geq 0$ for all $k \in \N$, hence by \eqref{diamestimate} and \eqref{capestimate2} we find
    \[ \text{diam}(E_k) \lesssim P(E_k)^{n-1} \lesssim \Pb(E_k)^{n-1} \leq \cF_{\alpha, \beta, Q}(E_k)^{n-1} \lesssim 1+Q^{2(n-1)}. \]
    We now focus on the case $\alpha=n$. Here we also have to show that  \eqref{minproblem} admits a finite lower bound since $\cI_{n}(E_k)$ could be negative as well. By inequalities \eqref{diamestimate} and \eqref{capestimate2},
    \begin{equation}\label{newtrick1} \text{diam}(E_k)^{\frac{1}{n-1}} \les \Pb(E_k).  \end{equation}
    By definition of $\cI_n(E)$ and since $|x-y| \leq \text{diam}(E_k)$ for all $x,y \in E_k$, we also obtain
    \begin{equation}\label{newtrick2}  -Q^2 \log(\text{diam}(E_k)) \leq Q^2 \cI_n(E_k).\end{equation}
    We sum up the two inequalities and find
    \begin{equation}\label{uniform bound diam} \text{diam}(E_k)^{\frac{1}{n-1}} -Q^2 \log(\text{diam}(E_k)) \leq \cF_{n, \beta, Q}(E_k) \lesssim 1+Q^2.\end{equation}
    Hence, there is $R>0$ such that $\text{diam}(E_k) \leq R$. Up to translation we may assume that $E_k \subset B_R$, so by monotonicity of $\cI_n$ we infer $\cI_{n}(E_k) \geq \cI_{n}(B_R)$. As a consequence,
    \[  \inf \left \{ \cF_{n, \beta, Q}(E) :\, E \subset H,\,E \text{ convex body},\,|E|=1\right \} >-\infty.\]
     Since $E_k$ is bounded, by Blaschke selection theorem \cite[Theorem 1.8.7]{schneider2013convex}, we can extract a not relabelled subsequence which converges in the Hausdorff and in the $L^1$ topologies to a convex body $E \subset H$ of volume $1$. Since $\beta \in (-1,1)$, the term $\Pb$ is lower semicontinuous with respect to the $L^1$ topology by \cite[Proposition 19.27]{maggi2012sets}. In addition, $\cIa$ is continuous with respect to Hausdorff convergence of convex bodies whenever the limit has positive volume, by Lemma \ref{contIa}. Hence, we combine these facts and conclude that $E$ is a minimizer of \eqref{minproblem}.
\end{proof}

\begin{remark}\label{uniform bound rem}
    Let $Q_0>0$ and let $E_Q$ be a minimizer of \eqref{minproblem} for some $Q \leq Q_0$. By inequality \eqref{uniform bound diam}
    and since $|E_Q|=1$, there exists $C_{Q_0}$ depending only on $Q_0$ for which we have the uniform bound $ 
    -C_{Q_0}\leq \cI_n(E_Q) \leq C_{Q_0}$ for all $Q \leq Q_0$. The lower bound follows from the previous proof, whereas the upper bound is obtained by testing $\cI_n(E_Q)$ with the uniform probability measure on $E_Q$.
\end{remark}

\begin{remark}\label{non-empty contact}
    Let $E$ be a minimizer of \eqref{minproblem}. Since $P(E) \geq \Pb(E)$ (which holds for all $\beta \in (-1,1)$) we may assume  that $E \cap \partial H \neq \emptyset$. In other words, the contact set between $E$ and $\partial H$ is always non-empty. This however does not exclude a priori that $\cH^{n-1}(\partial E \cap \partial H)=0$, see Corollary \ref{cor:poslenght}. \end{remark}

We now exploit the quantitative isoperimetric inequality for $\Pb$ to show that minimizers of \eqref{minproblem} converge to minimizers of $\Pb$ in the $L^1$ and in the Hausdorff topologies. Recall the notation introduced in Subsection \ref{capillarity}.

\begin{lemma}\label{hausdorffconv}
    Let $n \geq 2$ and $\alpha \in (0,n]$. Let $\beta \in (-1,1)$ and $E_Q$ be a minimizer of \eqref{minproblem} for $Q >0$. Up to translation we have that $E_Q \rightarrow B^{\beta}$ as $Q \rightarrow 0$ in the $L^1$ and in the Hausdorff topologies.
\end{lemma}
\begin{proof}
Let $B^{\beta}_x$ be the minimizer for the Fraenkel asymmetry $\alpha_{\beta}(E_Q)$ defined in \eqref{fraenkel}, up to translation we may assume that  $x=0$. Inequality \eqref{quantisop} implies that
\[ |E_Q \Delta B^{\beta}|^2 \lesssim \lt( \Pb(E_Q) - \Pb(B^{\beta}) \rt).\]
The minimality of $E_Q$ yields
\[ |E_Q \Delta B^{\beta}|^2 \lesssim \Pb(E_Q) - \Pb(B^{\beta}) \leq Q^2 \lt( \cIa(B^{\beta}) - \cIa(E_Q) \rt) \leq Q^2 (|\cIa(B^{\beta})|+ |\cIa(E_Q)|) \lesssim Q^2, \]
where  in the case $\alpha=n$, the last inequality follows from the uniform boundedness of the logarithmic energy of minimizers of $\mathcal{F}_{n, \beta, Q}$ from Remark \ref{uniform bound rem}. Taking the limit as $Q \rightarrow 0$, in both cases we infer that both
\begin{equation}\label{perconv}
    \Pb(E_Q) \longrightarrow \Pb(B^{\beta}) \qquad \text{and} \qquad E_Q \longrightarrow B^{\beta} \quad \text{in } L^1(\R^n). \end{equation}
Finally, the convergence $E_Q \rightarrow B^{\beta}$ as $Q \rightarrow 0$ holds in the Hausdorff sense too since we are considering convex bodies.
\end{proof}

\begin{remark}
    A consequence of the previous Lemma is the $L^1$ and Hausdorff convergence of the traces on $\partial H$, namely 
    \begin{equation}\label{convtrace}
         E_Q \cap \partial H  \rightarrow B^{\beta} \cap \partial H \qquad \text{in } L^1(\partial H) \quad \text{as } Q \rightarrow 0.
    \end{equation}
    Indeed, thanks to the Hausdorff convergence of $E_Q$ to $B^{\beta}$, there exists a convex body $K \subset B^{\beta} \cap \partial H$ such that $E_Q \cap \partial H \rightarrow K $ as $Q \rightarrow 0$ in the Hausdorff topology. Combining the first relation in \eqref{perconv} and the continuity of the perimeter under Hausdorff convergence of convex bodies we have that also $\cH^{n-1}(E_Q \cap \partial H) \rightarrow \cH^{n-1}(B^{\beta} \cap \partial H)$ as $Q \rightarrow 0$, therefore $K = B^{\beta} \cap \partial H$ and \eqref{convtrace} holds.
    \end{remark}

In Section \ref{sec3} we focus on the planar logarithmic case $ \alpha =n=2$. As a preliminary result, we show here that it is possible to relax the volume constraint in the minimization problem \eqref{minproblem} by adding a volume penalization to the functional $\mathcal{F}_{2,\beta, Q}$. This procedure is commonly used in isoperimetric type problems, see for example \cite{esposito2011remark, goldman2012volume, GolNovRuf16, Gol22, lamboley2023regularity}. For $\Lambda >0$ we define the functional
\[ \mathcal{F}_{2, \beta, Q, \Lambda}(E):= \mathcal{F}_{2,\beta, Q}(E) + \Lambda\lt| |E| - 1\rt|. \]
\begin{lemma}\label{volconstraint}
    Let $n=2$. For every $Q_0>0$, there exists $\Lambda_0>0$ depending only on $Q_0$ such that if $\Lambda>\Lambda_0$ and $Q \leq Q_0$, then minimizers of
    \begin{equation}\label{minproblemfree}
    \min \left \{ \cF_{2, \beta, Q, \Lambda}(E) : E \subset H,\,E \text{ convex body}\right \}
    \end{equation}
    are also minimizers of \eqref{minproblem} and vice-versa.
\end{lemma}
\begin{proof}
     We follow the proof of \cite[Lemma 4.1]{GolNovRuf16}.
 Fix $Q_0>0$ and $Q \leq Q_0$. Let $\{E_k\}_{k \in\N} \subset H$ be a minimizing sequence for \eqref{minproblemfree}. Let $B$ be a ball with $|B|=1$, for large enough $k$ we have $\cF_{2, \beta, Q, \Lambda}(E_k) \leq \cF_{2, \beta, Q, \Lambda}(B)$. Arguing as in the proof of Theorem \ref{theo:existence},  \eqref{newtrick1} and \eqref{newtrick2} imply
    \[ \begin{split}
        \text{diam}(E_k)- Q^2 \log(\text{diam}(E_k)) \les \cF_{2, \beta, Q, \Lambda}(E_k) \leq \cF_{2, \beta, Q, \Lambda}(B) = \cF_{2, \beta, Q}(B) \lesssim 1+Q^2_0. \end{split} \]
    Hence, as before there are $r,R>0$ depending only on $Q_0$ such that for large enough $k$ we have $r\leq \text{diam}(E_k) \leq R$. Up to translation we may assume that $E_k \subset B_R$, so by monotonicity of $\cI_2$ we infer $\cI_{2}(E_k) \geq \cI_{2}(B_R)$. As a consequence, we find
    \[  \inf \left \{ \cF_{2, \beta, Q, \Lambda}(E) : E \subset H,\,E \text{ convex body}\right \} = \lim_{k \rightarrow + \infty} \cF_{2, \beta, Q, \Lambda}(E_k) >-\infty.\]
    Up to passing to a subsequence, $\{E_k\}_{k \in\N}$ converges in the $L^1$ and in the Hausdorff topologies to a convex set $E \subset H$ whose diameter is also bounded by $R$. Under these conditions we have continuity of the volume term $\Lambda ||\cdot| - 1|$ and lower semicontinuity of $\Pb$ and $\cI_2$, so we conclude
    \[ \cF_{2, \beta, Q, \Lambda}(E) \leq \inf \left \{ \cF_{2, \beta, Q, \Lambda}(E) : E \subset H,\,E \text{ convex body}\right \}.\]
    To obtain the equality we only need to show that $|E|>0$. Let $C_{Q_0}>0$ be the constant depending only on $Q_0$ given by Remark \ref{uniform bound rem}. We claim that if
    \[ \Lambda > \Lambda_0:= 2P(B) + Q_0^2 \lt(\cI_2(B) + C_{Q_0} \rt) \]
    then we have $|E|>0$. Indeed, when $|E|=0$ the inequality 
    \[ \cF_{2, \beta, Q, \Lambda}(E) \leq \cF_{2, \beta, Q, \Lambda}(B) \leq 2P(B) + Q^2 \cI_2(B) \]
    implies
    \[ \begin{split} \Lambda & \leq 2 P(B) - \Pb(E) + Q^2 \lt( \cI_2(B) - \cI_2(E) \rt) \\ & \leq  2P(B) + Q_0^2 \lt(\cI_2(B) + C_{Q_0} \rt)=\Lambda_0. \end{split} \]
    Hence, if $\Lambda > \Lambda_0$ then $E$ is a convex body and it is a minimizer of \eqref{minproblemfree}. Notice that the minimum in \eqref{minproblemfree} is always less or equal than the minimum in \eqref{minproblem}. Hence, it remains to prove the converse inequality. Assume that $E$ is not a minimizer for $\mathcal{F}_{2, \beta, Q}$, in this case we set
    \[ \sigma:= ||E|-1|>0.\]
    From the uniform bound on the diameter of $E$ we deduce that $\Lambda \sigma$ is also bounded by a constant depending only on $Q_0$. Assuming that $|E|=1-\sigma$ (the other case is analogous), we define a competitor for the problem \eqref{minproblemfree} as
    \[ F:= \frac{1}{(1-\sigma)^{1/2}}E.\]
    By construction of the dilation we have that $|F|=1$, $\cI_2(F) \leq \cI_2(E)$  and
    \[ \Pb(F) = \frac{1}{(1-\sigma)^{1/2}} \Pb(E). \]
    Hence, a Taylor expansion yields
    \[ \begin{split} \Lambda \sigma & =  \cF_{2, \beta, Q, \Lambda}(E) - \cF_{2, \beta, Q}(E)  \leq \cF_{2, \beta, Q, \Lambda}(F) - \cF_{2, \beta, Q}(E) \\ & = \Pb(F) + Q^2 \cI_2(F) - \Pb(E) - Q^2 \cI_2(E) \\ & \leq \Pb(E) \lt( (1-\sigma)^{-1/2} - 1 \rt)
    \\ & \leq \Pb(E) \sigma.
    \end{split} \]
    Recalling that $\Pb(E)$ is bounded by a constant depending only on $Q_0$ since $E$ is a minimizer, we obtain $\Lambda \lesssim 1+Q^2$ as well. Therefore, if $\Lambda$ is large enough, me must have $\sigma=0$ or equivalently that $E$ is a minimizer of $\cF_{2, \beta, Q}$ too.
\end{proof}
\begin{remark}\label{remminprop}
   Let $E$ be a minimizer of \eqref{minproblem} and $F\subset H$ be a convex set, then by the previous lemma we have
    \begin{equation}\label{minimalityprop} \Pb(E) \leq \Pb(F) + Q^2 \lt( \cI_2(F)- \cI_2(E)\rt) + \Lambda |E \Delta F|. \end{equation}
    This $\Lambda$-minimality relation motivates the definitions introduced in Section \ref{sec5} and it will be the starting point of our analysis in Section \ref{sec3}.
\end{remark}
\end{section}

\begin{section}{\texorpdfstring{$\Lambda$}{TEXT}-Minimizers of capillarity shape functionals under convexity constraint}\label{sec5}

The purpose of this section is to introduce the notion of $\Lambda$-minimizers of capillarity shape functionals under convexity constraint in $\R^2$. We first present the definition of a capillarity shape functionals $J_{\beta}$ in the upper half-space $H=\{x \in\R^2:x_2 \geq 0\} \subset \R^2$. Then we turn to $\Lambda$-minimizers of $J_{\beta}$ under convexity constraint and we study some of their properties.

Let us first set some notation regarding convex bodies contained in $H$. If $E \subset H$ is a convex body with $E \cap \partial H \neq \emptyset$ then there are $P_1, P_2 \in E \cap \partial H$ (with possibly $P_1=P_2$) such that
\[ E \cap \partial H = [P_1, P_2].\]
We call $P_1,P_2$ the contact points of $E$ with $\partial H$ or just contact points of $E$ for short, see Figure \ref{pic contact set} for two examples. Let $\tau_1, \tau_2$ be tangent lines to $\partial E \cap H^\circ$ at the points $P_1, P_2$. We choose $\tau_1 \neq \tau_2$ if $P_1=P_2$ but $\partial E$ is not smooth at $P_1=P_2$. For $i \in \{1,2\}$ let $\nu_E^i$ be an outer unit normal to $E$ at $P_i$ which is orthogonal to $\tau_i$. Let $\nu_H=-e_2$ be the outer unit normal to $H$ and $\gamma_i \in (0,\pi]$ be the angle defined by the condition
\[ \cos \gamma_i:= - \nu_E^i \cdot \nu_H.\]
We call $\gamma_1,\gamma_2$ the contact angles of $E$ with $\partial H$, or just contact angles of $E$ for short. In the following we just denote by $\gamma$ any of them. We now present the definitions of capillarity shape functional and of $\Lambda$-minimizer of a capillarity shape functional under convexity constraint.

\begin{figure}
    \centering
\tikzset{every picture/.style={line width=0.75pt}} 

\begin{tikzpicture}[x=0.75pt,y=0.75pt,yscale=-1,xscale=1]

\draw [color={rgb, 255:red, 155; green, 155; blue, 155 }  ,draw opacity=1 ]   (37,95.89) -- (265.5,96.23) ;
\draw [color={rgb, 255:red, 208; green, 2; blue, 27 }  ,draw opacity=1 ]   (202.03,95.93) -- (202.14,133.78) ;
\draw [shift={(202.14,135.78)}, rotate = 269.83] [color={rgb, 255:red, 208; green, 2; blue, 27 }  ,draw opacity=1 ][line width=0.75]    (10.93,-3.29) .. controls (6.95,-1.4) and (3.31,-0.3) .. (0,0) .. controls (3.31,0.3) and (6.95,1.4) .. (10.93,3.29)   ;
\draw  [draw opacity=0][dash pattern={on 4.5pt off 4.5pt}] (173.53,95.03) .. controls (173.91,82.29) and (182.43,71.63) .. (193.99,68.2) -- (202.03,95.93) -- cycle ; \draw  [dash pattern={on 4.5pt off 4.5pt}] (173.53,95.03) .. controls (173.91,82.29) and (182.43,71.63) .. (193.99,68.2) ;  
\draw    (104.94,95.85) .. controls (43.4,32.47) and (177.1,11.37) .. (202.03,95.93) ;
\draw [color={rgb, 255:red, 208; green, 2; blue, 27 }  ,draw opacity=1 ]   (202.03,95.93) -- (245.96,82.13) ;
\draw [shift={(247.87,81.54)}, rotate = 162.56] [color={rgb, 255:red, 208; green, 2; blue, 27 }  ,draw opacity=1 ][line width=0.75]    (10.93,-3.29) .. controls (6.95,-1.4) and (3.31,-0.3) .. (0,0) .. controls (3.31,0.3) and (6.95,1.4) .. (10.93,3.29)   ;
\draw [color={rgb, 255:red, 0; green, 0; blue, 0 }  ,draw opacity=1 ]   (104.94,95.85) -- (202.03,95.93) ;
\draw  [draw opacity=0][dash pattern={on 4.5pt off 4.5pt}] (85.23,74.98) .. controls (90.35,70.01) and (97.29,66.96) .. (104.94,66.96) .. controls (120.69,66.96) and (133.46,79.89) .. (133.46,95.85) -- (104.94,95.85) -- cycle ; \draw  [dash pattern={on 4.5pt off 4.5pt}] (85.23,74.98) .. controls (90.35,70.01) and (97.29,66.96) .. (104.94,66.96) .. controls (120.69,66.96) and (133.46,79.89) .. (133.46,95.85) ;  
\draw [color={rgb, 255:red, 208; green, 2; blue, 27 }  ,draw opacity=1 ]   (104.94,95.85) -- (105.05,133.7) ;
\draw [shift={(105.06,135.7)}, rotate = 269.83] [color={rgb, 255:red, 208; green, 2; blue, 27 }  ,draw opacity=1 ][line width=0.75]    (10.93,-3.29) .. controls (6.95,-1.4) and (3.31,-0.3) .. (0,0) .. controls (3.31,0.3) and (6.95,1.4) .. (10.93,3.29)   ;
\draw [color={rgb, 255:red, 208; green, 2; blue, 27 }  ,draw opacity=1 ]   (104.94,95.85) -- (86.53,114.61) -- (77.97,123.34) ;
\draw [shift={(76.57,124.76)}, rotate = 314.47] [color={rgb, 255:red, 208; green, 2; blue, 27 }  ,draw opacity=1 ][line width=0.75]    (10.93,-3.29) .. controls (6.95,-1.4) and (3.31,-0.3) .. (0,0) .. controls (3.31,0.3) and (6.95,1.4) .. (10.93,3.29)   ;
\draw [color={rgb, 255:red, 155; green, 155; blue, 155 }  ,draw opacity=1 ]   (288.86,95.97) -- (469.83,96.06) ;
\draw  [draw opacity=0][dash pattern={on 4.5pt off 4.5pt}] (363.37,95.7) .. controls (363.45,83.86) and (373.03,74.28) .. (384.83,74.28) .. controls (387.96,74.28) and (390.93,74.95) .. (393.61,76.16) -- (384.83,95.85) -- cycle ; \draw  [dash pattern={on 4.5pt off 4.5pt}] (363.37,95.7) .. controls (363.45,83.86) and (373.03,74.28) .. (384.83,74.28) .. controls (387.96,74.28) and (390.93,74.95) .. (393.61,76.16) ;  
\draw [color={rgb, 255:red, 208; green, 2; blue, 27 }  ,draw opacity=1 ]   (384.83,95.85) -- (416.34,109.75) ;
\draw [shift={(418.17,110.56)}, rotate = 203.82] [color={rgb, 255:red, 208; green, 2; blue, 27 }  ,draw opacity=1 ][line width=0.75]    (10.93,-3.29) .. controls (6.95,-1.4) and (3.31,-0.3) .. (0,0) .. controls (3.31,0.3) and (6.95,1.4) .. (10.93,3.29)   ;
\draw [color={rgb, 255:red, 74; green, 144; blue, 226 }  ,draw opacity=1 ]   (318.83,45.23) -- (384.83,95.85) ;
\draw [color={rgb, 255:red, 74; green, 144; blue, 226 }  ,draw opacity=1 ]   (408.5,41.73) -- (384.83,95.85) ;
\draw  [draw opacity=0][dash pattern={on 4.5pt off 4.5pt}] (363.97,76.19) .. controls (369.75,69.91) and (378.22,66.27) .. (387.31,67.05) .. controls (403,68.4) and (414.61,82.38) .. (413.24,98.29) -- (384.83,95.85) -- cycle ; \draw  [dash pattern={on 4.5pt off 4.5pt}] (363.97,76.19) .. controls (369.75,69.91) and (378.22,66.27) .. (387.31,67.05) .. controls (403,68.4) and (414.61,82.38) .. (413.24,98.29) ;  
\draw [color={rgb, 255:red, 208; green, 2; blue, 27 }  ,draw opacity=1 ]   (384.83,95.85) -- (384.94,131.08) ;
\draw [shift={(384.95,133.08)}, rotate = 269.82] [color={rgb, 255:red, 208; green, 2; blue, 27 }  ,draw opacity=1 ][line width=0.75]    (10.93,-3.29) .. controls (6.95,-1.4) and (3.31,-0.3) .. (0,0) .. controls (3.31,0.3) and (6.95,1.4) .. (10.93,3.29)   ;
\draw [color={rgb, 255:red, 208; green, 2; blue, 27 }  ,draw opacity=1 ]   (384.83,95.85) -- (362.72,124.64) ;
\draw [shift={(361.5,126.23)}, rotate = 307.53] [color={rgb, 255:red, 208; green, 2; blue, 27 }  ,draw opacity=1 ][line width=0.75]    (10.93,-3.29) .. controls (6.95,-1.4) and (3.31,-0.3) .. (0,0) .. controls (3.31,0.3) and (6.95,1.4) .. (10.93,3.29)   ;
\draw [color={rgb, 255:red, 0; green, 0; blue, 0 }  ,draw opacity=1 ]   (352.5,71.08) -- (384.83,95.85) ;
\draw    (352.5,71.08) .. controls (282.17,17.2) and (418.17,19.38) .. (384.83,95.85) ;
\draw [color={rgb, 255:red, 74; green, 144; blue, 226 }  ,draw opacity=1 ]   (60.5,49.98) -- (104.94,95.85) ;
\draw [color={rgb, 255:red, 74; green, 144; blue, 226 }  ,draw opacity=1 ]   (186.25,41.98) -- (202.03,95.93) ;

\draw (247.56,65.64) node [anchor=north west][inner sep=0.75pt]  [font=\footnotesize] [align=left] {$\nu _{E}^{2}$};
\draw (205.07,135.67) node [anchor=north west][inner sep=0.75pt]  [font=\footnotesize] [align=left] {$\nu _{H}$};
\draw (138.86,52.61) node [anchor=north west][inner sep=0.75pt]  [font=\small] [align=left] {$E$};
\draw (235.52,27.62) node [anchor=north west][inner sep=0.75pt]  [font=\small] [align=left] {$H$};
\draw (44.31,34.12) node [anchor=north west][inner sep=0.75pt]  [font=\footnotesize]  {$\tau _{1}$};
\draw (175.62,24.42) node [anchor=north west][inner sep=0.75pt]  [font=\footnotesize]  {$\tau _{2}$};
\draw (109.9,102.53) node [anchor=north west][inner sep=0.75pt]  [font=\footnotesize]  {$P_{1}$};
\draw (182.56,102.16) node [anchor=north west][inner sep=0.75pt]  [font=\footnotesize]  {$P_{2}$};
\draw (103.46,139.2) node [anchor=north west][inner sep=0.75pt]  [font=\footnotesize] [align=left] {$\nu _{H}$};
\draw (57.37,126.11) node [anchor=north west][inner sep=0.75pt]  [font=\footnotesize] [align=left] {$\nu _{E}^{1}$};
\draw (105.84,76.28) node [anchor=north west][inner sep=0.75pt]  [font=\footnotesize]  {$\gamma _{1}$};
\draw (180.51,77.61) node [anchor=north west][inner sep=0.75pt]  [font=\footnotesize]  {$\gamma _{2}$};
\draw (419.41,110.1) node [anchor=north west][inner sep=0.75pt]  [font=\footnotesize] [align=left] {$\nu _{E}^{2}$};
\draw (354.05,45.41) node [anchor=north west][inner sep=0.75pt]  [font=\small] [align=left] {$E$};
\draw (457.04,26.75) node [anchor=north west][inner sep=0.75pt]  [font=\small] [align=left] {$H$};
\draw (302.83,30.58) node [anchor=north west][inner sep=0.75pt]  [font=\footnotesize]  {$\tau _{1}$};
\draw (407.65,27.05) node [anchor=north west][inner sep=0.75pt]  [font=\footnotesize]  {$\tau _{2}$};
\draw (386.43,102.77) node [anchor=north west][inner sep=0.75pt]  [font=\footnotesize]  {$P_{1}$};
\draw (377.32,135.67) node [anchor=north west][inner sep=0.75pt]  [font=\footnotesize] [align=left] {$\nu _{H}$};
\draw (343.56,122.24) node [anchor=north west][inner sep=0.75pt]  [font=\footnotesize] [align=left] {$\nu _{E}^{1}$};
\draw (343.69,76.41) node [anchor=north west][inner sep=0.75pt]  [font=\footnotesize]  {$\gamma _{2}$};
\draw (414.19,63.74) node [anchor=north west][inner sep=0.75pt]  [font=\footnotesize]  {$\gamma _{1}$};

\end{tikzpicture}
    \caption{Contact sets and contact angles}
    \label{pic contact set}
\end{figure}
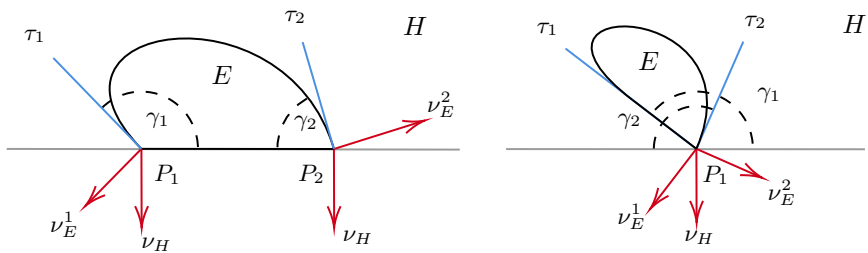

\begin{definition}
    Let $\beta \in (-1,1)$ and $V$ be a $\R$-valued shape functional defined on subsets of $\R^2$, then $J_{\beta}$ is a capillarity shape functional in $H=\{x \in\R^2:x_2 \geq 0\} \subset \R^2$ if its of the form
    \[ J_{\beta}(E) = \Pb(E) + V(E).\]
\end{definition}

\begin{definition}\label{lambdamin}
    A convex body $E \subset H$ is a $\Lambda$-minimizer of $J_{\beta}$ under convexity constraint if for every convex body $F \subset H$ we have
    \[ J_{\beta}(E) \leq J_{\beta}(F) + \Lambda |E \Delta F|. \]
\end{definition}

Notice that the definition of $\Lambda$-minimality is equivalent to asking that for every $F \subset H$ convex body 
\begin{equation}\label{altdeflambdamin} \Pb(E) \leq \Pb(F) + (V(F)- V(E)) + \Lambda |E \Delta F|.\end{equation}
\begin{remark}
 Let us point out that if $E$ is a $\Lambda$-minimizer of $P+V$ and $d(E,\partial H)\gg P(E)$, then it is also a $\Lambda$-minimizer of $J_{\beta}$. We can thus not assume in general that $\partial E\cap \partial H \neq \emptyset$ for $\Lambda-$minimizers of $J_\beta$.
\end{remark}
The results of this section will be proven under a specific assumption on the perturbative term $V$. In particular, we prescribe an upper bound on the term $V(F)-V(E)$ in \eqref{altdeflambdamin} for suitable choices of variations $F$ obtained by cutting-off $E$. These variations are  constructed in \cite{GolNovRuf16}, and require that the point $x\in \partial E$ around which the cut is made and $\eps>0$ satisfy \eqref{intersect2points} below. While this point was slightly overlooked in \cite{GolNovRuf16} we give here sufficient conditions under which it holds.

\begin{lemma}\label{convex circle}
    Let $E \subset \R^2$ be a convex body, and  $x \in \partial E$. Assume that for some $\eps_x,t_x>0$ there exists a coordinate system $\{e,e^\perp\}$ for which $e$ is tangent to $\partial E$ at $x$ and
    \begin{equation}\label{graphprop}
     (\partial E-x)\cap \{ x'e + t e^\perp : \, |x'|\le \eps_x,\, t\le t_x\}=\{ x'e + u(x') e^\perp: \, |x'|\le \eps_x \}
    \end{equation}
    for some convex function $u$. Then, for every $0<\eps \le \min(\eps_x,t_x)$,  we have for some $x_1^{\eps}\neq x_2^{\eps} \in \partial E$,
     \begin{equation}\label{intersect2points} \partial E \cap \partial B_{\eps}(x) = \{x_1^{\eps}, x_2^{\eps}\}. \end{equation}
As a consequence, for all $x\in \partial E$ there exists $\eps_x>0$ such that for all $\eps \leq \eps_x$ \eqref{intersect2points} holds. Moreover, if
\begin{equation}\label{smalloscillation}
 \sup_{y,y'\in \partial E \cap  B_{\eps_x}(x)}\lt|\nu_E(y)-\nu_E(y')\rt|<\sqrt{2},
\end{equation}
then \eqref{intersect2points} holds for $\eps\le \eps_x$.

\end{lemma}
\begin{proof}
    We assume without loss of generality that $x=0$. If $\partial E$ is a graph  of a convex function $u$ over $\R e$ in $\{x'e+ te^\perp: \, |x'|\le \eps_x, t\le t_x\}$, then $E\subset \{ x\cdot e^\perp\ge 0\}$  and $u$ is  minimal at $x'=0$. It is therefore decreasing for $x'<0$ and increasing for $x'>0$. Fix now $\eps\le \min(\eps_x,t_x)$. The set $\partial B_\eps \cap \{x\cdot e^\perp\ge 0\}$ is then the  graph of a concave function which is increasing for $x'<0$ and increasing for $x'>0$. By a simple intermediate value argument we see that \eqref{intersect2points}  indeed holds. Assume now that the tangent cone to $E$ in $0$ is given by
    \[ C(e^\perp,\theta)=\{ x' e+ te^\perp:\, \tan(\theta/2) t\ge |x'|\}, \]
    where $\theta\in(0,\pi]$ and $e^\perp \in \S^1$. By definition of tangent cone, $e$ is tangent to $\partial E$ in $0$ and if $\eps_x$ is small enough,
    \[ C(e^\perp,\theta/2)\cap \{ x'e+ t e^\perp: \, |x'|\le \eps_x, t\le \bar{t}\, \}\subset E. \]
    Here $\bar t$ is a constant which depends on $E$ through its Lipschitz character but not on $x$ so we may assume that $\eps_x\le \bar t$. Since $\min(|x'|/\tan(\theta/4)),\bar t)\in E$ for every $|x'|\le \eps_x$, by convexity of $E$ we see that \eqref{graphprop} holds. Finally if we assume \eqref{smalloscillation} then \eqref{graphprop} also holds with $t_x=\eps_x$.\end{proof}
    
    \begin{remark}
     Let us point out that in general the threshold $\eps_x$ from Lemma \ref{convex circle} is not uniform in $x$. To see this it is enough to consider the case of a cone of opening angle strictly less than $\pi/2$ so that $\eps_x\to 0$ as $x$ converges to the tip of the cone.
    \end{remark}

Let $E \subset \R^2$ be a fixed convex body. In \cite[Section 4]{GolNovRuf16} the first two authors introduced a class of competitors $\cC_E$ obtainable from $E$ by a cutting procedure. We repeat the construction  for the sake of completeness, see Figure \ref{cutting angle} for an example. Let $x \in \partial E$ and $\eps_x>0$ be given by Lemma \ref{convex circle}. For every $\eps \leq \eps_x$ we call $\ell_{x,\eps}$ the line passing through $x_1^{\eps/3}$ and $x_2^{\eps/3}$. If $x \notin \ell_{x,\eps}$ we denote by $H_{x, \eps}$ the closed half-space with boundary $\ell_{x,\eps}$ and not containing $x$. If instead $x \in \ell_{x, \eps}$ then $\ell_{x,\eps}$ is the supporting line of $E$ at $x$ and we denote by $H_{x, \eps}$ the closed half-space with boundary $\ell_{x,\eps}$ and containing $E$. We define
\[ E_{x,\eps}:= E \cap H_{x,\eps}. \]
\begin{figure}
    \centering
    \tikzset{every picture/.style={line width=0.75pt}} 

\begin{tikzpicture}[x=0.75pt,y=0.75pt,yscale=-1,xscale=1]

\draw    (166.42,54.65) -- (244.39,50.3) -- (340.98,56.83) ;
\draw    (79.83,99.33) .. controls (113.5,63) and (131.71,59.55) .. (166.42,54.65) ;
\draw    (340.98,56.83) .. controls (383.8,60.7) and (391,62.3) .. (405.4,86.3) ;
\draw  [draw opacity=0] (177.27,59.29) .. controls (177.57,61.38) and (177.75,63.5) .. (177.82,65.66) -- (117.63,67.63) -- cycle ; \draw  [color={rgb, 255:red, 155; green, 155; blue, 155 }  ,draw opacity=1 ] (177.27,59.29) .. controls (177.57,61.38) and (177.75,63.5) .. (177.82,65.66) ;  
\draw    (66.88,69.13) -- (423.4,59.3) ;
\draw [color={rgb, 255:red, 155; green, 155; blue, 155 }  ,draw opacity=1 ]   (244.39,50.3) -- (372.13,60.63) ;
\draw  [draw opacity=0] (311.92,62.34) .. controls (311.9,61.77) and (311.89,61.2) .. (311.89,60.63) .. controls (311.89,59.07) and (311.95,57.53) .. (312.07,56) -- (372.13,60.63) -- cycle ; \draw  [color={rgb, 255:red, 155; green, 155; blue, 155 }  ,draw opacity=1 ] (311.92,62.34) .. controls (311.9,61.77) and (311.89,61.2) .. (311.89,60.63) .. controls (311.89,59.07) and (311.95,57.53) .. (312.07,56) ;  
\draw [color={rgb, 255:red, 208; green, 2; blue, 27 }  ,draw opacity=1 ]   (79.83,99.33) .. controls (92,85.75) and (107,73.25) .. (117.63,67.63) ;
\draw [color={rgb, 255:red, 208; green, 2; blue, 27 }  ,draw opacity=1 ]   (372.13,60.63) .. controls (393.5,64.75) and (398,75.25) .. (405.4,86.3) ;
\draw [color={rgb, 255:red, 155; green, 155; blue, 155 }  ,draw opacity=1 ]   (244.39,50.3) -- (117.63,67.63) ;
\draw [color={rgb, 255:red, 208; green, 2; blue, 27 }  ,draw opacity=1 ]   (117.63,67.63) -- (372.13,60.63) ;

\draw (400.09,88.33) node [anchor=north west][inner sep=0.75pt]  [font=\footnotesize] [align=left] {$\partial E$};
\draw (181.87,70.37) node [anchor=north west][inner sep=0.75pt]  [font=\footnotesize] [align=left] {$\gamma _{x,\varepsilon }$};
\draw (239.8,36.23) node [anchor=north west][inner sep=0.75pt]  [font=\footnotesize]  {$x$};
\draw (107.1,47.25) node [anchor=north west][inner sep=0.75pt]  [font=\footnotesize]  {$x_{1}^{\varepsilon /3}$};
\draw (367.25,41.57) node [anchor=north west][inner sep=0.75pt]  [font=\footnotesize]  {$x_{2}^{\varepsilon /3}$};
\draw (426.8,52.1) node [anchor=north west][inner sep=0.75pt]  [font=\footnotesize]  {$\ell _{x,\varepsilon }$};
\draw (287.19,72.36) node [anchor=north west][inner sep=0.75pt]  [font=\footnotesize,color={rgb, 255:red, 208; green, 2; blue, 27 }  ,opacity=1 ] [align=left] {$E_{x,\varepsilon }$};

    \end{tikzpicture}
    \caption{Competitor obtained by cutting}
    \label{cutting angle}
\end{figure}
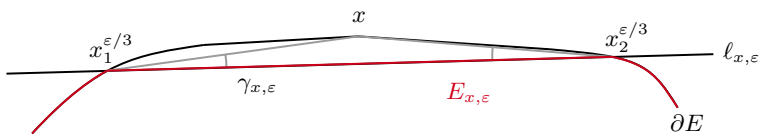
By construction $E_{x,\eps}$ is a convex body and $E_{x,\eps} \subset E$. We hence define the class of sets
\[ \cC_E:=\lt\{ F \subset E: F= E_{x,\eps} \text{ for some } x \in \partial E \text{ and } \eps \leq \eps_x\rt\}.\]
Finally, let $\gamma_{x,\eps}$ be the angle between the segment $[x, x_1^{\eps/3}]$ and $\ell_{x,\eps}$ (notice that it coincides with the angle between $[x, x_2^{\eps/3}]$ and $\ell_{x,\eps}$). From now on, whenever it is clear from the context we drop the dependence of $E_{x,\eps}$ and $\gamma_{x, \eps}$ on $x \in \partial E$ and we just write $E_{\eps}$ and $\gamma_\eps$. We are ready to state the assumption we make on the term $V$ of a capillarity shape functional $J_{\beta}$.  

\begin{assumption}\label{assloclip}
    The functional $V$ is monotone decreasing with respect to set inclusion, namely $V(F) \geq V(E)$ whenever $F \subset E$. In addition, for every convex body $E \subset \R^2$ there exist $q>1$, $s> 1$ and $C>0$ such that for all $E_{x,\eps} \in \cC_E$ we have
    \begin{equation}\label{ineqloclip} V(E_{x,\eps}) - V(E) \leq C \lt( \eps^{q}+ \eps^s\gamma_\eps \rt). \end{equation}
\end{assumption}

\begin{remark}\label{assumption ok}
    The logarithmic energy $\cI_2$ satisfies Assumption \ref{assloclip}, as proved in \cite[Theorem 4.4]{GolNovRuf16}.
    Indeed, the shape functional $\cI_2$ is monotone decreasing with respect to set inclusion by definition. Moreover, by Theorem \ref{convexint} we have  $\mu_E = f_E \llcorner \partial E$ for some $ f_E \in L^p(\partial E)$ with $p>2$. By \textit{Step 3} of the proof of \cite[Theorem 4.4]{GolNovRuf16}, there exists $C>0$ depending on $\| f_E \|_{L^{p}(\partial E)}$ and on the Lipschitz character of $\partial E$ such that for all $E_{x,\eps} \in \cC_E$ relation \eqref{ineqloclip} is satisfied with $C>0$ and $q=s= 2 - 2/p > 1$. \end{remark}

\begin{remark}\label{lamboley prunier}
    A more standard notion of $\Lambda$-minimality would be to require that for $F\subset E$ with $|E\backslash F|\ll 1$ we have
    \begin{equation}\label{hypstrongLambda}
     V(F)-V(E)\le \Lambda |E\backslash F|^\alpha
    \end{equation}
    for some $1/2 < \alpha \leq 1$, see \cite{lamboley2023regularity}. By \textit{Step 1} of the proof of \cite[Theorem 4.4]{GolNovRuf16}, if $F = E_{x,\eps} \in \cC_E$ then $|E \setminus E_{x,\eps}|\lesssim \eps^2 \gamma_\eps$ so that \eqref{hypstrongLambda} implies
    \[ V(F)-V(E)\les \Lambda \eps^{2\alpha} \gamma_\eps^{\alpha}. \]
Therefore \eqref{hypstrongLambda} implies \eqref{ineqloclip}. The motivation to work with \eqref{ineqloclip} is to include the case of $V=\cI_2$. Notice also that, as opposed to \cite{lamboley2023regularity}, here we also need to use variations which are not contained in $E$.
\end{remark}

For $x \in H$, we set $d_{\partial H}(x):= \min\{|x-y|: y \in \partial H\}$. We follow the strategy of \cite[Section 4]{GolNovRuf16} to show how Assumption \ref{assloclip} implies regularity of the boundary of $\Lambda$-minimizers of $J_{\beta}$ away from the contact points.

\begin{theorem}\label{C1,etareg}
    Let $J_{\beta}$ be a capillarity shape functional in $H$ satisfying Assumption \ref{assloclip} and $E \subset H$ be a $\Lambda$-minimizer of $J_{\beta}$ under convexity constraint and set
    \[ \eta:=\min(s-1,(q-1)/2,1) \in (0,1]. \]
    Then, $\partial E \cap \{x_2>0\}$ is of class $C^{1,\eta}_{loc}$.
\end{theorem}
\begin{proof}
Fix $x\in \partial E\cap \{x_2>0\}$. Let  $\eps_x>0$ be such that \eqref{graphprop} holds for $\eps\le \eps_x$. For  $\eps \leq \min(\eps_x, d_{\partial H}(x))$, we consider the set $E_{x,\eps} \in \cC_E$ constructed above. We test the $\Lambda$-minimality of $E$ using $E_{x,\eps}$ as a competitor and obtain
    \begin{equation}\label{cutminimality} \Pb(E) - \Pb(E_{x,\eps}) \leq (V(E_{x,\eps})- V(E)) + \Lambda |E \Delta E_{x,\eps}|. \end{equation}
    We now estimate each difference in terms of $\eps$ and $\gamma_{\eps}$. Notice that by construction $E \setminus E_{x,\eps} \subset \{x_2>0\}$, so the capillarity contribution in $\Pb$ plays no role. By \textit{Step 1} and \textit{Step 2} of the proof of \cite[Theorem 4.4]{GolNovRuf16} we have
    \begin{equation}\label{volper} |E \setminus E_{x,\eps}|= |E| - |E_{x,\eps}| \simeq \eps^2\gamma_{\eps} \qquad \text{and} \qquad \Pb(E)- \Pb(E_{x,\eps}) =P(E) - P(E_{x,\eps}) \gtrsim \eps \gamma_{\eps}^2.  \end{equation}
    Therefore \eqref{cutminimality} and Assumption \ref{assloclip} yield
    \[ \eps \gamma_{\eps}^2 \les \eps^{q}+ \gamma_\eps \eps^s + \Lambda  \eps^2 \gamma_\eps. \]
    A disjunction of cases together with \eqref{volper} implies
    \begin{equation}\label{decaygamma}
    |E \setminus E_{x,\eps}| \les (1+\Lambda) \eps^{2+\eta}.
    \end{equation} 
    We now rule out the presence of angles in $\partial E \cap \{x_2>0\}$. This will show that  $\partial E\cap \{x_2>0\}$ is $C^1$. Arguing  by contradiction we  assume that there is $x \in \partial E \cap \{x_2>0\}$ for which the tangent cone $\cup_{\lambda >0} \lambda\lt(E - x\rt)$ spans an angle $\theta \in (0, \pi)$. The presence of the spanned angle $\theta$ implies that
    \begin{equation*}\label{anglecond}  |E \setminus E_{x,\eps}| \gtrsim \eps^2 \theta. \end{equation*}
    For $\eps$ small enough this gives a contradiction with \eqref{decaygamma}. Finally since $\partial E \cap \{x_2>0\}$ is of class $C^1$, for every $\delta>0$ there is $\eps_\delta$ such that \eqref{smalloscillation} holds for $x\in \partial E \cap \{x_2>\delta\}$ and $\eps_x=\eps_\delta$. Since now \eqref{decaygamma} holds uniformly in $x\in \partial E \cap \{x_2>\delta\}$, thanks to \cite[Lemma 4.2]{GolNovRuf16} we conclude that $E \cap \{x_2>\delta\}$ is of class $C^{1,\eta}$. \end{proof}

\begin{remark}\label{C1,1reg}
    If $q\geq 3$ and $s \geq 2$ we obtain from Theorem \ref{C1,etareg} the optimal $C^{1,1}_{loc}$ regularity.
\end{remark}

\begin{remark}\label{uniform estimates}
    Let us notice that if $\{E_k\}_{k \in \N}$ is a sequence of $\Lambda$-minimizers converging in the Hausdorff sense to a convex body $E \subset H$, with $\partial E \cap \{x_2>0\}$ of class $C^1$, then for every $\delta>0$, if $k$ is large enough (depending on $\delta$), then \eqref{smalloscillation} holds for every $x\in \partial E_k \cap \{x_2>\delta\}$  with $\eps_x=\eps_\delta$ uniform.  In turn this implies that $\partial E_k \cap \{x_2>\delta\}$ are uniformly $C^{1,\eta}$. This applies for instance to \cite[Theorem 4.4]{GolNovRuf16} and Theorem \ref{younglog1}.
\end{remark}

We now assume that  $E \cap \partial H  \neq \emptyset$ and investigate whether Young's law $\cos \gamma = \beta$ holds.  We begin by studying the inequality $\cos \gamma \geq \beta$.

\begin{theorem}\label{young2}
    Let $J_{\beta}$ be a capillarity shape functional in $H$ for which the term $V$ is decreasing with respect to set inclusion. Let $E \subset H$ be a $\Lambda$-minimizer of $J_{\beta}$ under convexity constraint, $0$ be a contact point of $E$ with contact angle $\gamma$ and $C_0$ be the tangent cone of $E$ at $0$. Assume that there is $R>0$ such that for all $r \leq R$ we have $E \cap B_r \neq C_0 \cap B_r$, then $\cos \gamma \geq \beta$.
\end{theorem}
\begin{proof}
    Let $\R e$ be the tangent line to $\partial E\cap H$ in $0$. Since $E$ does not locally coincide with its tangent cone, up to reflection, there exists a convex function $u: \R^+\to \R^+$ of class $C^1$ with $u(0)=u'(0)=0$ (here $u'(0)$ denotes the right derivative), $u(x)>0$ for $x>0$  and such that
    \[ \partial E\cap H\cap B_R=\{x e+ u(x) e^{\perp}, x\ge 0\}\cap B_R. \]
    From now on we use $(e,e^\perp)$ as basis of $\R^2$. The three cases $\gamma \in (0,\pi/2)$, $\gamma \in (\pi/2,\pi]$ and $\gamma = \pi/2$ are analogous, so we treat them at once pointing out the differences if needed. See Figure \ref{picYounggeq} for the corresponding constructions. If $\gamma \neq \pi/2$ then $\partial H=\{ (x,x\tan \gamma): x \in \R \}$, so that
    \[ \begin{split}
        & \text{if } \gamma \in (0, \pi/2) \qquad E\cap B_R=\{ (x,y): x\tan \gamma\ge y\ge u(x)\} \cap B_R,\\ & \text{if } \gamma \in (\pi/2, \pi] \qquad E\cap B_R=\lt\{ (x,y): y\geq \max(x\tan \gamma, u(x)) \rt\} \cap B_R.
    \end{split}\]
    If instead $\gamma = \pi/2$ then $\partial H = \{(0,y):y \in \R\}$ and hence
    \[ E\cap B_R=\{ (x,y): x\geq 0,\, y \ge u(x)\} \cap B_R. \]
    Since $u'(0)=0$ with $u(x)>0$ for $x>0$, there exists a sequence $v_k, x_{v_k}\to 0$ with $u(x_{v_k})=v_k$. To lighten notation we write $v$ for $v_k$.  We set $u_v:=u(x_v)$ and
    \[l(x):=u_v+v(x - x_v),\]
    so that $\ell := \{(x, l(x)):x \in \R \}$ is the tangent line to $\partial E$ at $(x_v, u_v)$, see Figure \ref{picYounggeq}. If $\gamma \neq \pi/2$ we define $x_0$ by the condition $l(x_0)= x_0 \tan \gamma$, hence the point $(x_0, x_0\tan\gamma)$ is the intersection between $\ell$ and $\partial H$. If instead $\gamma = \pi/2$ then we set $x_0=0$ and here $(0, l(0))$ is the intersection between $\ell$ and $\partial H$. We take as competitor the set $F \subset H$ such that $F \setminus B_R = E \setminus B_R$ and, see Figure \ref{picYounggeq},
    \begin{alignat*}{2} & \text{if } \gamma \in (0, \pi/2) \qquad && F \cap B_R = \lt\{ (x,y): \begin{cases} x\tan \gamma \geq y \geq l & \text{if } x \leq x_v, \\ x\tan \gamma \geq y \geq u & \text{if } x \geq x_v. \end{cases} \rt \} \cap B_R, \\ & \text{if } \gamma = \pi/2 \qquad && F \cap B_R = \lt\{ (x,y): \begin{cases} y \geq l & \text{if } 0 \leq x \leq x_v, \\ y \geq u & \text{if } x \geq x_v. \end{cases} \rt \} \cap B_R, \\ & \text{if } \gamma \in (\pi/2, \pi] \qquad && F \cap B_R = \lt\{ (x,y): \begin{cases} y \geq \max(x\tan \gamma, l) & \text{if } x \leq x_v, \\ y \geq u & \text{if } x \geq x_v. \end{cases} \rt \} \cap B_R. \end{alignat*}
    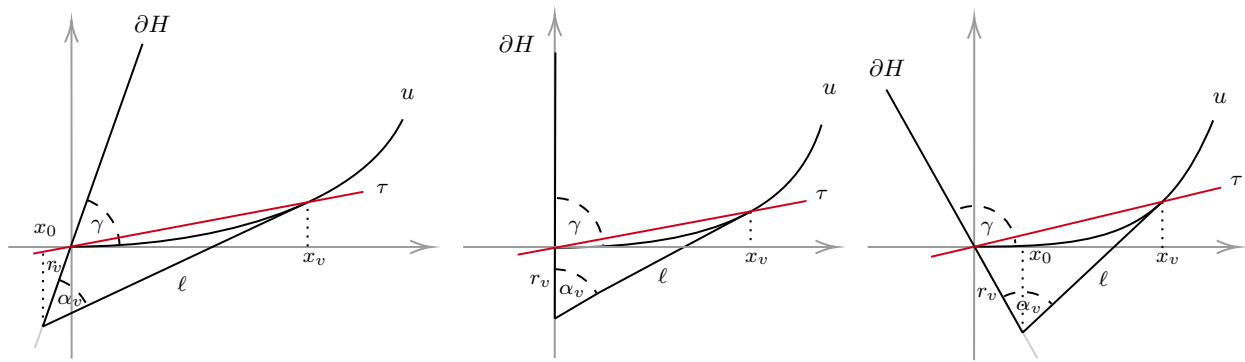
\begin{figure}
        \centering
        \tikzset{every picture/.style={line width=0.75pt}} 
\begin{tikzpicture}[x=0.75pt,y=0.75pt,yscale=-1,xscale=1]

\draw [color={rgb, 255:red, 155; green, 155; blue, 155 }  ,draw opacity=1 ]   (52.56,194.03) -- (52.48,25.48) ;
\draw [shift={(52.48,23.48)}, rotate = 89.97] [color={rgb, 255:red, 155; green, 155; blue, 155 }  ,draw opacity=1 ][line width=0.75]    (10.93,-3.29) .. controls (6.95,-1.4) and (3.31,-0.3) .. (0,0) .. controls (3.31,0.3) and (6.95,1.4) .. (10.93,3.29)   ;
\draw [color={rgb, 255:red, 155; green, 155; blue, 155 }  ,draw opacity=1 ]   (20.62,137.72) -- (232.24,137.72) ;
\draw [shift={(234.24,137.72)}, rotate = 180] [color={rgb, 255:red, 155; green, 155; blue, 155 }  ,draw opacity=1 ][line width=0.75]    (10.93,-3.29) .. controls (6.95,-1.4) and (3.31,-0.3) .. (0,0) .. controls (3.31,0.3) and (6.95,1.4) .. (10.93,3.29)   ;
\draw    (88.32,35.18) -- (52.1,137.55) ;
\draw [color={rgb, 255:red, 155; green, 155; blue, 155 }  ,draw opacity=1 ]   (296.22,193.75) -- (296.22,22.22) ;
\draw [shift={(296.22,20.22)}, rotate = 90] [color={rgb, 255:red, 155; green, 155; blue, 155 }  ,draw opacity=1 ][line width=0.75]    (10.93,-3.29) .. controls (6.95,-1.4) and (3.31,-0.3) .. (0,0) .. controls (3.31,0.3) and (6.95,1.4) .. (10.93,3.29)   ;
\draw    (296.54,39.53) -- (296.22,173.89) ;
\draw [color={rgb, 255:red, 155; green, 155; blue, 155 }  ,draw opacity=1 ]   (507.71,194.84) -- (507.79,23.59) ;
\draw [shift={(507.79,21.59)}, rotate = 90.03] [color={rgb, 255:red, 155; green, 155; blue, 155 }  ,draw opacity=1 ][line width=0.75]    (10.93,-3.29) .. controls (6.95,-1.4) and (3.31,-0.3) .. (0,0) .. controls (3.31,0.3) and (6.95,1.4) .. (10.93,3.29)   ;
\draw [color={rgb, 255:red, 155; green, 155; blue, 155 }  ,draw opacity=1]   (454.04,137.72) -- (639,137.72) ;
\draw [shift={(641,137.72)}, rotate = 180] [color={rgb, 255:red, 155; green, 155; blue, 155 }  ,draw opacity=1 ][line width=0.75]    (10.93,-3.29) .. controls (6.95,-1.4) and (3.31,-0.3) .. (0,0) .. controls (3.31,0.3) and (6.95,1.4) .. (10.93,3.29)   ;
\draw    (507.91,137.55) .. controls (575.78,137.72) and (604.73,130.92) .. (628.32,73.8) ;
\draw    (463.47,58.3) -- (480.32,88.34) -- (507.91,137.55) ;
\draw [color={rgb, 255:red, 155; green, 155; blue, 155 }  ,draw opacity=0.5 ]   (33.99,188.73) -- (52.1,137.55) ;
\draw [color={rgb, 255:red, 155; green, 155; blue, 155 }  ,draw opacity=0.5 ]   (507.91,137.55) -- (539.56,193.75) ;
\draw  [draw opacity=0][dash pattern={on 4.5pt off 4.5pt}] (60.38,114.07) .. controls (69.88,117.5) and (76.68,126.67) .. (76.73,137.44) -- (52.1,137.55) -- cycle ; \draw  [dash pattern={on 4.5pt off 4.5pt}] (60.38,114.07) .. controls (69.88,117.5) and (76.68,126.67) .. (76.73,137.44) ;  
\draw  [draw opacity=0][dash pattern={on 4.5pt off 4.5pt}] (297.32,113.03) .. controls (309.77,113.49) and (319.82,123.71) .. (320.59,136.53) -- (296.44,138.09) -- cycle ; \draw  [dash pattern={on 4.5pt off 4.5pt}] (297.32,113.03) .. controls (309.77,113.49) and (319.82,123.71) .. (320.59,136.53) ;  
\draw  [draw opacity=0][dash pattern={on 4.5pt off 4.5pt}] (500.76,118.28) .. controls (503.74,117.18) and (507.02,116.75) .. (510.4,117.16) .. controls (520.29,118.34) and (527.7,126.29) .. (528.51,135.78) -- (507.91,137.55) -- cycle ; \draw  [dash pattern={on 4.5pt off 4.5pt}] (500.76,118.28) .. controls (503.74,117.18) and (507.02,116.75) .. (510.4,117.16) .. controls (520.29,118.34) and (527.7,126.29) .. (528.51,135.78) ;  
\draw    (52.1,137.55) .. controls (118.62,137.72) and (197.01,119.5) .. (219.55,73.26) ;
\draw  [dash pattern={on 0.84pt off 2.51pt}]  (171.57,137.45) -- (171.57,115.14) ;
\draw  [dash pattern={on 0.84pt off 2.51pt}]  (394.99,119.77) -- (394.99,137.99) ;
\draw [color={rgb, 255:red, 0; green, 0; blue, 0 }  ,draw opacity=1 ]   (602.66,114.87) -- (532.18,180.96) ;
\draw  [dash pattern={on 0.84pt off 2.51pt}]  (602.66,138.26) -- (602.66,127.93) -- (602.66,114.87) ;
\draw    (296.44,138.09) .. controls (380.12,137.99) and (418.11,116.78) .. (430.77,75.98) ;
\draw [color={rgb, 255:red, 0; green, 0; blue, 0 }  ,draw opacity=1 ]   (394.99,119.77) -- (318.66,160.68) -- (296.22,173.89) ;
\draw    (38.04,177.85) -- (171.57,115.14) ;
\draw    (507.91,137.55) -- (532.18,180.96) ;
\draw    (38.04,177.85) -- (52.1,137.55) ;
\draw [color={rgb, 255:red, 155; green, 155; blue, 155 }  ,draw opacity=1 ]   (250,137.72) -- (438.68,137.72) ;
\draw [shift={(438.68,137.72)}, rotate = 180] [color={rgb, 255:red, 155; green, 155; blue, 155 }  ,draw opacity=1 ][line width=0.75]    (10.93,-3.29) .. controls (6.95,-1.4) and (3.31,-0.3) .. (0,0) .. controls (3.31,0.3) and (6.95,1.4) .. (10.93,3.29)   ;
\draw  [draw opacity=0][dash pattern={on 4.5pt off 4.5pt}] (46.32,154.37) .. controls (52.16,156.49) and (56.99,160.77) .. (59.84,166.26) -- (38.04,177.85) -- cycle ; \draw  [dash pattern={on 4.5pt off 4.5pt}] (46.32,154.37) .. controls (52.16,156.49) and (56.99,160.77) .. (59.84,166.26) ;  
\draw  [dash pattern={on 0.84pt off 2.51pt}]  (38.04,177.85) -- (38.2,138.24) ;
\draw  [draw opacity=0][dash pattern={on 4.5pt off 4.5pt}] (295.8,148.98) .. controls (295.94,148.97) and (296.08,148.97) .. (296.22,148.97) .. controls (305.26,148.97) and (313.16,153.9) .. (317.45,161.24) -- (296.22,173.89) -- cycle ; \draw  [dash pattern={on 4.5pt off 4.5pt}] (295.8,148.98) .. controls (295.94,148.97) and (296.08,148.97) .. (296.22,148.97) .. controls (305.26,148.97) and (313.16,153.9) .. (317.45,161.24) ;  
\draw  [dash pattern={on 0.84pt off 2.51pt}]  (532.18,180.96) -- (532.2,137.84) ;
\draw  [draw opacity=0][dash pattern={on 4.5pt off 4.5pt}] (522.79,162.68) .. controls (525.62,161.24) and (528.83,160.42) .. (532.23,160.42) .. controls (537.93,160.42) and (543.08,162.72) .. (546.81,166.42) -- (532.18,180.96) -- cycle ; \draw  [dash pattern={on 4.5pt off 4.5pt}] (522.79,162.68) .. controls (525.62,161.24) and (528.83,160.42) .. (532.23,160.42) .. controls (537.93,160.42) and (543.08,162.72) .. (546.81,166.42) ;  
\draw [color={rgb, 255:red, 208; green, 2; blue, 27 }  ,draw opacity=1 ]   (52.1,137.55) -- (171.57,115.14) ;
\draw [color={rgb, 255:red, 208; green, 2; blue, 27 }  ,draw opacity=1 ]   (296.44,138.09) -- (394.99,119.77) ;
\draw [color={rgb, 255:red, 208; green, 2; blue, 27 }  ,draw opacity=1 ]   (507.91,137.55) -- (602.66,114.87) ;
\draw [color={rgb, 255:red, 208; green, 2; blue, 27 }  ,draw opacity=1 ]   (33.2,141.06) -- (52.1,137.55) ;
\draw [color={rgb, 255:red, 208; green, 2; blue, 27 }  ,draw opacity=1 ]   (171.57,115.14) -- (199.74,109.85) ;
\draw [color={rgb, 255:red, 208; green, 2; blue, 27 }  ,draw opacity=1 ]   (277.54,141.6) -- (296.44,138.09) ;
\draw [color={rgb, 255:red, 208; green, 2; blue, 27 }  ,draw opacity=1 ]   (394.99,119.77) -- (423.16,114.48) ;
\draw [color={rgb, 255:red, 208; green, 2; blue, 27 }  ,draw opacity=1 ]   (486.02,142.66) -- (507.91,137.55) ;
\draw [color={rgb, 255:red, 208; green, 2; blue, 27 }  ,draw opacity=1 ]   (602.66,114.87) -- (632.33,107.71) ;

\draw (216.85,56.34) node [anchor=north west][inner sep=0.75pt]  [font=\footnotesize]  {$u$};
\draw (429.84,54.05) node [anchor=north west][inner sep=0.75pt]  [font=\footnotesize]  {$u$};
\draw (626.69,58.19) node [anchor=north west][inner sep=0.75pt]  [font=\footnotesize]  {$u$};
\draw (82.58,19.35) node [anchor=north west][inner sep=0.75pt]  [font=\footnotesize]  {$\partial H$};
\draw (266.53,29.4) node [anchor=north west][inner sep=0.75pt]  [font=\footnotesize]  {$\partial H$};
\draw (453.21,41.37) node [anchor=north west][inner sep=0.75pt]  [font=\footnotesize]  {$\partial H$};
\draw (60.84,122.58) node [anchor=north west][inner sep=0.75pt]  [font=\scriptsize]  {$\gamma$};
\draw (301.59,123.22) node [anchor=north west][inner sep=0.75pt]  [font=\scriptsize]  {$\gamma$};
\draw (509.68,123.31) node [anchor=north west][inner sep=0.75pt]  [font=\scriptsize]  {$\gamma$};
\draw (44,160.24) node [anchor=north west][inner sep=0.75pt]  [font=\scriptsize]  {$\alpha_v$};
\draw (38.3,143.14) node [anchor=north west][inner sep=0.75pt]  [font=\scriptsize]  {$r\hspace{-2pt}_v$};
\draw (32,124.6) node [anchor=north west][inner sep=0.75pt]  [font=\scriptsize]  {$x_{0}$};
\draw (167.6,139.4) node [anchor=north west][inner sep=0.75pt]  [font=\scriptsize]  {$x_{v}$};
\draw (282.1,150.54) node [anchor=north west][inner sep=0.75pt]  [font=\scriptsize]  {$r_v$};
\draw (297,155.44) node [anchor=north west][inner sep=0.75pt]  [font=\scriptsize]  {$\alpha_v$};
\draw (527.5,163.4) node [anchor=north west][inner sep=0.75pt]  [font=\scriptsize]  {$\alpha_v$};
\draw (507.6,156.84) node [anchor=north west][inner sep=0.75pt]  [font=\scriptsize]  {$r_v$};
\draw (390.8,139) node [anchor=north west][inner sep=0.75pt]  [font=\scriptsize]  {$x_{v}$};
\draw (598,139) node [anchor=north west][inner sep=0.75pt]  [font=\scriptsize]  {$x_{v}$};
\draw (534,138.2) node [anchor=north west][inner sep=0.75pt]  [font=\scriptsize]  {$x_{0}$};
\draw (204.8,104.2) node [anchor=north west][inner sep=0.75pt]  [font=\scriptsize]  {$\tau$};
\draw (635.6,100.2) node [anchor=north west][inner sep=0.75pt]  [font=\scriptsize]  {$\tau$};
\draw (426.4,107.2) node [anchor=north west][inner sep=0.75pt]  [font=\scriptsize]  {$\tau$};
\draw (346.8,146.84) node [anchor=north west][inner sep=0.75pt]  [font=\scriptsize]  {$\ell $};
\draw (104.4,150.84) node [anchor=north west][inner sep=0.75pt]  [font=\scriptsize]  {$\ell $};
\draw (569.2,148.44) node [anchor=north west][inner sep=0.75pt]  [font=\scriptsize]  {$\ell $};
    \end{tikzpicture}   
    \caption{Contact angle $\gamma \in (0, \pi/2)$, $\gamma = \pi/2$ and $\gamma \in (\pi/2, \pi)$.}
        \label{picYounggeq}
    \end{figure}Notice that $F \supset E$ is a convex set because $u' \geq l'=v$ for $x\geq x_v$. We now introduce more notation. Let $\alpha_v < \gamma$ be the angle between $\partial H$ and $\ell$, so that $\gamma - \alpha_v$ is the acute angle between $\ell$ and $\R e$ and we have 
    \[ v = \tan (\gamma - \alpha_v). \]
    Thanks to the above expression when $\gamma \neq \pi/2$ we write $x_0$ as
    \[ x_0= \frac{u_v -vx_v}{\tan \gamma - v} = \frac{u_v -vx_v}{\tan \gamma - \tan(\gamma - \alpha_v)} = x_v\lt( \frac{u_v}{x_v}-v \rt) \frac{\cos(\gamma - \alpha_v)}{\sin \alpha_v} \cos \gamma. \]
    We now set $r_v=\mathcal{H}^1(\partial H\cap (F\setminus E))$. By construction of $F$ for all $\gamma \in (0, \pi]$ we have
    \[ -x_0= r_v \cos \gamma, \]
    therefore we infer
    \begin{equation}\label{tracediff} r_v= x_v\lt(v -
    \frac{u_v}{x_v}\rt) \frac{\cos(\gamma - \alpha_v)}{\sin \alpha_v}.\end{equation}
    Notice that by construction $r_v \rightarrow 0$ as $v \rightarrow 0$. Finally, let $\tau:=\{(x, t(x)): x \in \R\}$ be the line passing through the origin and $(x_v, u_v)$, so that
    \[ t(x)= \frac{u_v}{x_v} x  \qquad \text{for } x \in \R.  \]
    Since $E \subset F$ we have $|E \Delta F|=|E \setminus F|$ and $V(E) \geq V(F)$ so the $\Lambda$-minimality property of $E$ yields
    \[  P_\beta(E) \le P_\beta(F) + \Lambda |F \setminus E|, \]
    which rewrites as
    \begin{equation}\label{outermin}
        \beta (P(F, \partial H) - P(E, \partial H)) \leq  P(F, H^\circ) - P(E, H^\circ) + \Lambda |F \setminus E|.
    \end{equation}
    We estimate each of the terms separately. First, by definition we have $P(F, \partial H) - P(E, \partial H) = r_v$. Then, we notice that
    \[  P(E, H^\circ) \geq \int_0^{x_v} \lt(1 + (t')^2\rt)^{1/2} = x_v \lt( 1 + \frac{u_v^2}{x_v^2}\rt)^{1/2}. \]
    Hence for all $\gamma \in (0, \pi]$ we find the upper bound
    \[ \begin{split} P(F, H^\circ) - P(E, H^\circ) & \leq (x_v - x_0)\lt(1 +v^2\rt)^{1/2}  - x_v \lt( 1 + \frac{u_v^2}{x_v^2}\rt)^{1/2} \\ & \leq 
     - x_0 \lt(1 +v^2\rt)^{1/2} +  x_v  \lt(v - \frac{u_v}{x_v}\rt)\lt(v + \frac{u_v}{x_v}\rt) \\ & = r_v \cos \gamma (1 + o(1)) +  r_v \lt(v + \frac{u_v}{x_v}\rt) \frac{\sin \alpha_v}{\cos (\gamma - \alpha_v)} \\ & = r_v \cos \gamma (1 + o(1)) + o(r_v),
    \end{split} \]
    where in the last line we used the facts that $\sin \alpha_v/\cos (\gamma - \alpha_v) \rightarrow \sin \gamma$ and $u_v/x_v \rightarrow 0$ as $v \rightarrow 0$. It thus remains to treat the volume term. Since the cases $\gamma \in (0, \pi/2)$, $\gamma = \pi/2$ and $\gamma \in (\pi/2, \pi]$ are analogous we consider only the first one. When $\gamma \in (0, \pi/2)$ we have
    \[ \begin{split} |F \setminus E| & = \int_{x_0}^{x_v} \lt( \min(x\tan \gamma, u(x)) - l(x) \rt)dx \leq  \int_{x_0}^{x_v} (t(x) - l(x))dx \\ & = \int_{x_0}^{x_v} \lt( vx_v - u_v - \lt( v - \frac{u_v}{x_v}\rt)x \rt)\, dx = \int_{x_0}^{x_v}\lt( v - \frac{u_v}{x_v}\rt)(x_v - x)\,dx \\ & =  \frac{1}{2} \lt( v - \frac{u_v}{x_v}\rt)(x_v - x_0)^2 = o(r_v), \end{split}\]
    where the last equality is a consequence of the explicit expression \eqref{tracediff}. Plugging all the estimates back in \eqref{outermin} we find
    \[ \beta r_v \leq r_v \cos \gamma  \lt( 1 + o(1) \rt) + o(r_v) = r_v \cos \gamma  \lt( 1 + o(1) \rt),  \]
    so dividing by $r_v$ and sending $v \rightarrow 0$ we conclude $\beta \leq \cos \gamma$. \end{proof}
\begin{corollary}\label{cor:poslenght}
  Let $J_{\beta}$ be a capillarity shape functional in $H$ satisfying Assumption \ref{assloclip}. Let $E \subset H$ be a $\Lambda$-minimizer of $J_{\beta}$ under convexity constraint such that $\partial E\cap \partial H\neq \emptyset$. Then, we have $\cH^1(\partial E \cap \partial H)>0$.
\end{corollary}
\begin{proof}
We argue by contradiction and assume, up to translation, that $ \partial E\cap \partial H=\{0\}$. Arguing as in the proof of Theorem \ref{C1,etareg}, we obtain that $\partial E$ is $C^1$ (notice that the competitors $E_{0,\eps}$ are all contained in $\{x_2>0\}$). Thus, $\partial H$ is the only support line to $\partial E$ at $0$ and the contact angle of $E$ with $\partial H$ is $\gamma = \pi$. Moreover, for every $r>0$ we have
    \[ E \cap B_r \neq H \cap B_r. \]
    Theorem \ref{young2} implies that $-1= \cos \pi \geq \beta$, but this is a contradiction with $\beta \in (-1, 1)$.
\end{proof}

We now prove that the hypothesis the $E$ does not locally coincide with its tangent cone is necessary to get the inequality $\cos \gamma\ge \beta$.

\begin{lemma}\label{examplenotyoung}
    There exist $\Lambda>0$ and $E \subset H$ such that $E$ is a $\Lambda$-minimizer of $\Pb$ for all $\beta \in (0,1)$ and its contact angle is $\gamma=\pi/2$.
\end{lemma}

\begin{proof}
    Let $L,\Lambda \gg 1$ and $B_1^{+}:= \overline B_1 \cap H$. We consider the set (see Figure \ref{epscompetitors})
    \[ E^{+}:= ([-1, 1] \times [0,L]) \cup (B_1^{+} + L e_2) \] and prove that it is a $\Lambda$-minimizer of $\Pb$ under convexity constraint for any $\beta \in (0,1)$.  From now on we fix $\beta \in (0,1)$. We claim that for any convex body $F^+ \subset H$ we have
    \begin{equation}\label{counterex} \Pb(E^+ )\leq  \Pb(F^+) + \Lambda|E^+\Delta F^+|. \end{equation}
    Assume this is not the case, namely there exists a convex body $F^+ \subset H$ such that
    \begin{equation}\label{countcontr} \Pb (F^+) + \Lambda |E^+ \Delta F^+| < \Pb(E^+). \end{equation}
    As a consequence we find $\Lambda |E^+ \Delta F^+| < \Pb(E^+) \simeq L$. We can thus assume that $|E^+ \Delta F^+| \ll 1$ up to taking $\Lambda \gg L$. Let now $\sigma : \R^2 \rightarrow \R^2$ be the symmetry $\sigma (x_1, x_2) = (x_1, -x_2)$ with respect to the line $\partial H$ and $E^-:=\sigma (E^+)$, $F^-:=\sigma (F^+)$. We consider the sets
    \[ E:= E^+ \cup E^-  \qquad \text{ and } F= F^+ \cup F^- \]
    for which there holds
    \begin{equation}\label{doublesize} |E \Delta F| = 2|E^+ \!\Delta F^+|, \quad P(E) = 2 P(E^+, H^\circ) \quad \text{and} \quad P(F) = 2P(F^+, H^\circ).\end{equation}
     The set $E$ is a $C^{1,1}$ convex body with $|E|= 4L + \pi$, $P(E)= 4L + 2\pi$, and
    \[ \sup_{x \in \partial E} \kappa_E = 1 \leq \frac{P(E)}{|E|} =: \lambda_{E}. \]
    The constant $\lambda_E$ is called the Cheeger constant of $E$, see \cite{caselles2007uniqueness}. By \cite[Lemma 3 and Theorem 4]{bellettini2002total} we have that $E$ is calibrable, namely there exists a vector field $T: \R^2 \rightarrow \R^2$ such that 
    \[ \begin{cases}
        \text{div }T = \lambda_{E} & \text{on } E \\
        \text{div }T = 0 & \text{on } \R^2 \setminus E \\ |T(x)| \leq 1 & \text{a.e. } x \in \R^2 \\ T\cdot \nu_{E} = 1 & \text{on } \partial E.
    \end{cases}\]
    We set
    \[
    \{\nu_E =\pm \nu_F\}:= \{x\in \partial E \cap \partial F:\, \nu_E(x)= \pm \nu_F(x)\}.
    \]
    Thanks to the above properties, applying the divergence theorem and using \cite[Theorem 16.3]{maggi2012sets} we find
    \[ \begin{split} \lambda_{E} |E \setminus F| & = \int_{E \setminus F} \text{div } T  = \int_{\partial E \cap F^c} T \cdot \nu_{E} + \int_{\{\nu_E=-\nu_F\}} T \cdot \nu_{E} - \int_{\partial F\cap E^\circ} T \cdot \nu_{F} \\ & \geq P(E, F^c) + \cH^1(\{\nu_E=-\nu_F\}) - P(F, E^\circ) \end{split}\]
    and
    \[ \begin{split} 0  & = \int_{F \setminus E} \text{div } T  = \int_{\partial F \cap E^c} T \cdot \nu_{F} + \int_{\{\nu_E=-\nu_F\}} T \cdot \nu_{F} - \int_{\partial E\cap F^\circ} T \cdot \nu_{E} \\ &  \leq P(F, E^c) + \cH^1(\{\nu_E=-\nu_F\})- P(E,F^\circ). \end{split} \]
    We sum the two inequalities, add the quantity $\cH^1(\{\nu_E=\nu_F\})$ on both sides and we get
    \[ P(E) - P(F) \leq \lambda_{E} |E \setminus F| \leq \lambda_{E} |E \Delta F| \]
    which by \eqref{doublesize} writes as
    \[ P(E^+, H^\circ) \leq P(F^+, H^\circ) +  \lambda_{E} |E^+ \Delta F^+|.\]
    Taking into account the capillarity contribution to $\Pb$ we obtain
    \begin{equation}\label{calib} \Pb(E^+) \leq \Pb(F^+) + \beta \lt( P(F^+ ,\partial H) - P(E^+, \partial H)\rt) + \lambda_{E} |E^+ \Delta F^+|. \end{equation}
    If $\eps := P(F^+ ,\partial H) - P(E^+, \partial H) \leq 0$ then we get a contradiction with \eqref{countcontr} as $\Lambda \gg \lambda_E$. If instead $\eps >0$, by combining \eqref{countcontr} and \eqref{calib} we obtain 
    \[ \Pb (F^+) + \Lambda |E^+ \Delta F^+|  \leq \Pb(F^+) + \beta \eps + \lambda_{E} |E^+ \Delta F^+| \]
    Thanks to the assumption $\Lambda \gg \lambda_E$  we get
    \[ \frac{\Lambda}{2} |E^+ \Delta F^+| < \beta \eps. \]
    We now claim that under our hypotheses we have $|E^+ \Delta F^+| \geq L\eps /8$. This would be enough to conclude, as $\Lambda L \eps < 16 \beta \eps$ implies $\eps = 0$ (by the fact that $\Lambda L \gg  \beta$) which is a contradiction. Thus, it remains to prove the claim. Since $|E^+ \Delta F^+| \ll 1$, there exist $\eps_+, \eps_-$ such that $\eps_++\eps_-= \eps$ and for which
    \[ \partial F^+ \cap \partial H = [-1 - \eps_-, 1 + \eps_+]. \]
    Without loss of generality we assume $\eps_+ > |\eps_-|\geq0$. See Figure \ref{epscompetitors} for two examples, corresponding respectively to $\eps_-<0$ and $\eps_+>0$.
    \begin{figure}
        \centering
        \tikzset{every picture/.style={line width=0.75pt}} 

\begin{tikzpicture}[x=0.75pt,y=0.75pt,yscale=-1,xscale=1]

\draw  [fill={rgb, 255:red, 208; green, 2; blue, 27 }  ,fill opacity=0.25 ] (156.2,151.2) -- (180.4,216.14) -- (156.2,216.14) -- cycle ;
\draw  [draw opacity=0][fill={rgb, 255:red, 229; green, 229; blue, 229 }  ,fill opacity=0.5 ] (84.22,84.86) .. controls (85.33,64.74) and (101.02,48.79) .. (120.22,48.79) .. controls (139.33,48.79) and (154.97,64.61) .. (156.2,84.6) -- (120.22,87.15) -- cycle ; \draw   (84.22,84.86) .. controls (85.33,64.74) and (101.02,48.79) .. (120.22,48.79) .. controls (139.33,48.79) and (154.97,64.61) .. (156.2,84.6) ;  
\draw  [draw opacity=0][fill={rgb, 255:red, 229; green, 229; blue, 229 }  ,fill opacity=0.5 ] (84.22,84.6) -- (156.2,84.6) -- (156.2,216.14) -- (84.22,216.14) -- cycle ;
\draw [color={rgb, 255:red, 155; green, 155; blue, 155 }  ,draw opacity=1 ]   (47.2,216) -- (203.4,216) ;
\draw    (156.2,84.6) -- (156.2,216.6) ;
\draw    (84.2,84.6) -- (84.2,216.6) ;
\draw    (74.2,156.6) .. controls (89.4,22.2) and (116.6,32.2) .. (155,60.6) ;
\draw    (74.2,156.6) -- (96,215.6) ;
\draw [color={rgb, 255:red, 208; green, 2; blue, 27 }  ,draw opacity=1 ]   (156.2,150.6) -- (156.2,216.14) ;
\draw [color={rgb, 255:red, 208; green, 2; blue, 27 }  ,draw opacity=1 ]   (156.2,150.6) -- (180.4,215.8) ;
\draw [color={rgb, 255:red, 208; green, 2; blue, 27 }  ,draw opacity=1 ]   (156.2,216.14) -- (180.4,215.8) ;
\draw    (155,60.6) -- (180.4,215.8) ;
\draw  [draw opacity=0][fill={rgb, 255:red, 229; green, 229; blue, 229 }  ,fill opacity=0.5 ] (258.32,84.9) .. controls (259.43,64.78) and (275.12,48.83) .. (294.32,48.83) .. controls (313.43,48.83) and (329.07,64.64) .. (330.3,84.63) -- (294.32,87.18) -- cycle ; \draw   (258.32,84.9) .. controls (259.43,64.78) and (275.12,48.83) .. (294.32,48.83) .. controls (313.43,48.83) and (329.07,64.64) .. (330.3,84.63) ;  
\draw  [draw opacity=0][fill={rgb, 255:red, 229; green, 229; blue, 229 }  ,fill opacity=0.5 ] (258.32,84.63) -- (330.3,84.63) -- (330.3,216.17) -- (258.32,216.17) -- cycle ;
\draw [color={rgb, 255:red, 155; green, 155; blue, 155 }  ,draw opacity=1 ]   (221.3,216.03) -- (377.5,216.03) ;
\draw    (330.3,84.63) -- (330.3,216.63) ;
\draw    (258.3,84.63) -- (258.3,216.63) ;
\draw    (256.3,70.6) .. controls (267.9,67.4) and (280.3,65.8) .. (294.32,87.18) ;
\draw    (256.3,70.6) -- (249.9,176.2) -- (255.17,215.83) ;
\draw    (294.32,87.18) -- (354.5,215.83) ;
\draw  [color={rgb, 255:red, 208; green, 2; blue, 27 }  ,draw opacity=1 ][fill={rgb, 255:red, 208; green, 2; blue, 27 }  ,fill opacity=0.25 ] (330.3,149.57) -- (306.1,84.63) -- (330.3,84.63) -- cycle ;

\draw (112.8,118.2) node [anchor=north west][inner sep=0.75pt]  [font=\footnotesize]  {$E^{+}$};
\draw (152.4,219.2) node [anchor=north west][inner sep=0.75pt]  [font=\scriptsize]  {$1$};
\draw (74.4,218.8) node [anchor=north west][inner sep=0.75pt]  [font=\scriptsize]  {$-1$};
\draw (68,79.2) node [anchor=north west][inner sep=0.75pt]  [font=\scriptsize]  {$L$};
\draw (142,136.4) node [anchor=north west][inner sep=0.75pt]  [font=\scriptsize]  {$\frac{L}{2}$};
\draw (170.3,219.37) node [anchor=north west][inner sep=0.75pt]  [font=\scriptsize]  {$1+\varepsilon _{+}$};
\draw (286.9,118.23) node [anchor=north west][inner sep=0.75pt]  [font=\footnotesize]  {$E^{+}$};
\draw (326.5,219.23) node [anchor=north west][inner sep=0.75pt]  [font=\scriptsize]  {$1$};
\draw (248.5,218.83) node [anchor=north west][inner sep=0.75pt]  [font=\scriptsize]  {$-1$};
\draw (242.1,79.23) node [anchor=north west][inner sep=0.75pt]  [font=\scriptsize]  {$L$};
\draw (331.7,136.43) node [anchor=north west][inner sep=0.75pt]  [font=\scriptsize]  {$\frac{L}{2}$};
\draw (344.4,219.4) node [anchor=north west][inner sep=0.75pt]  [font=\scriptsize]  {$1+\varepsilon _{+}$};
\draw (152.4,36.2) node [anchor=north west][inner sep=0.75pt]  [font=\footnotesize]  {$F^{+}$};
\draw (298.8,71) node [anchor=north west][inner sep=0.75pt]  [font=\scriptsize]  {$\varepsilon _{+}$};
\draw (346.9,179) node [anchor=north west][inner sep=0.75pt]  [font=\footnotesize]  {$F^{+}$};

    \end{tikzpicture}
    \caption{Two cases.}
    \label{epscompetitors}
    \end{figure}
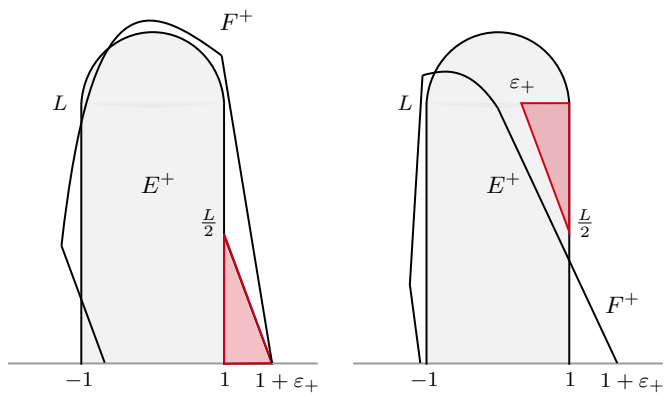
    If the set $F^+ \setminus E^+$ contains a triangle leaning on $\partial H$ with base of length $\eps_+$ and height $L/2$, see again Figure \ref{epscompetitors}, then we have
    \[ \frac{L}{4}\eps_+ \leq |F^+ \setminus E^+| \leq |E^+ \Delta F^+|.\]
    If this is not the case, setting
    \[ x \in \argmin \lt\{ x_2 : x=(x_1, x_2) \in H \cap \partial E^+ \cap \partial F^+, x_1 \geq 0 \rt\} \]
    there holds $x_2 \leq L/2$. By convexity of $F$ we have that $E^+ \setminus F^+$ contains a triangle with base of length $\eps_+$ and height $L/2$ so as before we conclude $L\eps_+/4 \leq |E^+ \Delta F^+|$. Arguing in the same way for $|\eps_-|$ we find $L|\eps_-|/4 \leq |E^+ \Delta F^+|$, thus summing up the two inequalities we obtain
    \[ \frac{L}{4}\eps \leq \frac{L}{4}(\eps_+ + |\eps_-|) \leq 2 |E^+ \Delta F^+| \]
    and the claim is proven.
\end{proof}

Let $E$ be a $\Lambda$-minimizer of $J_\beta$ under convexity constraint with non-empty contact set, we prove the inequality $\cos \gamma \leq \beta$. The proof is similar to that of Theorem \ref{young2} and it involves the construction of competitors from the interior.

\begin{theorem}\label{young1}
     Let $J_{\beta}$ be a capillarity shape functional in $H$ satisfying Assumption \ref{assloclip}. Let $E \subset H$ be a $\Lambda$-minimizer of $J_{\beta}$ under convexity constraint and $0$ be a contact point of $E$ with contact angle $\gamma$, then $\cH^1(\partial E \cap \partial H)>0$ and $\cos \gamma \leq \beta$.
\end{theorem}

\begin{proof}
    The first part of the statement follows from Corollary \ref{cor:poslenght}. We thus focus on the inequality $\cos \gamma \leq \beta$. Without loss of generality we assume $\gamma < \pi$, otherwise there is nothing to prove. Let $\R e$ be the tangent line to $\partial E\cap H$ in $0$. Up to reflection, there exist $R>0$ and a convex function $u: \R^+\to \R^+$  with $u(0)=u'(0)=0$ (here $u'(0)$ denotes the right derivative) such that
    \[ \partial E\cap H\cap B_R=\{x e+ u(x) e^{\perp}, x\ge 0\}\cap B_R. \]
    From now on we use $(e,e^\perp)$ as basis of $\R^2$ and we set $u(x)=0$ when $x \leq 0$. We distinguish the three cases $\gamma \in (0,\pi/2)$, $\gamma \in (\pi/2,\pi)$ and $\gamma = \pi/2$, see Figure \ref{picYoungleq}. We focus first on $\gamma \neq \pi/2$. In these two cases we have $\partial H=\{ (x,x\tan \gamma): x \in \R \}$, so that
        \[ \begin{split}
        & \text{if } \gamma \in (0, \pi/2) \qquad E\cap B_R=\{ (x,y): x\tan \gamma\ge y\ge u(x)\} \cap B_R,\\ & \text{if } \gamma \in (\pi/2, \pi) \qquad E\cap B_R=\lt\{ (x,y): y\geq \max(x\tan \gamma, u(x)) \rt\} \cap B_R.
    \end{split}\]
    Let $\alpha > \gamma$ be fixed for the moment. For $\bar x>0$ we set
    \begin{equation}\label{cutting line}
        l(x):=\tan(\alpha-\gamma)(\bar x-x),
    \end{equation}
    so the line $\ell := \{(x, l(x)): x \in \R \}$ passes through the point $(\bar x,0)$ and makes an angle $\alpha$ with $\partial H$. We consider as a competitor the set $F \subset H$ such that $F \setminus B_R=E \setminus B_R $ and
    \[ \begin{split}
        & \text{if } \gamma \in (0, \pi/2) \qquad F \cap B_R = \lt\{ (x,y): x\tan \gamma\ge y\ge \max (u(x), l(x)) \rt\}\cap B_R, \\ & \text{if } \gamma \in (\pi/2, \pi) \qquad F\cap B_R=\lt\{ (x,y): y\geq \max(x\tan \gamma, u(x), l(x)) \rt\} \cap B_R.
    \end{split}\]
    By definition $F \subset E$ and $F$ is convex. We now introduce some more notation. Let $x_0$ be such that $x_0 \tan \gamma = l(x_0)$, so the point $(x_0, x_0 \tan \gamma)$ is the intersection between $\ell$ and $\partial H$. By definition \eqref{cutting line} of $l$ we have
    \begin{equation}\label{point x_0} x_0= \frac{ \tan (\alpha - \gamma)}{\tan \gamma + \tan (\alpha -\gamma)} \bar x = \frac{\sin(\alpha-\gamma)\cos \gamma}{\sin \alpha} \bar x. \end{equation}
    We set $r=\mathcal{H}^1(\partial H\cap (E\setminus F))$. From the relation $x_0= r \cos \gamma$ we infer
    \[ r= \frac{\sin(\alpha-\gamma)}{\sin \alpha} \bar x.\]
    We finally define $x_1$ by the condition $u(x_1)=l(x_1)$. By construction $u(x) = o(x)$ as $x \rightarrow 0$ so we have
    \begin{equation}\label{point x_1} x_1 = \bar x + o_{\alpha}(\bar x ) \qquad \text{as } \bar x \rightarrow 0, \end{equation}
    where we use the notation $o_{\alpha}(\bar x)$, $O_{\alpha}(\bar x)$ to highlight the dependence on $\alpha$.
    \begin{figure}
        \centering
\tikzset{every picture/.style={line width=0.75pt}} 

\begin{tikzpicture}[x=0.75pt,y=0.75pt,yscale=-1,xscale=1]

\draw [color={rgb, 255:red, 155; green, 155; blue, 155 }  ,draw opacity=1 ]   (47.87,184.14) -- (47.87,21.14) ;
\draw [shift={(47.87,19.14)}, rotate = 90] [color={rgb, 255:red, 155; green, 155; blue, 155 }  ,draw opacity=1 ][line width=0.75]    (10.93,-3.29) .. controls (6.95,-1.4) and (3.31,-0.3) .. (0,0) .. controls (3.31,0.3) and (6.95,1.4) .. (10.93,3.29)   ;
\draw [color={rgb, 255:red, 155; green, 155; blue, 155 }  ,draw opacity=1 ]   (20.75,155.89) -- (203.5,155.89) ;
\draw [shift={(205.5,155.89)}, rotate = 180] [color={rgb, 255:red, 155; green, 155; blue, 155 }  ,draw opacity=1 ][line width=0.75]    (10.93,-3.29) .. controls (6.95,-1.4) and (3.31,-0.3) .. (0,0) .. controls (3.31,0.3) and (6.95,1.4) .. (10.93,3.29)   ;
\draw    (151.66,32.14) -- (63.17,137.61) -- (47.97,155.73) ;
\draw    (29.29,92.03) -- (181.53,173.4) ;
\draw [color={rgb, 255:red, 155; green, 155; blue, 155 }  ,draw opacity=1 ]   (248.87,184.14) -- (248.87,21.14) ;
\draw [shift={(248.87,19.14)}, rotate = 90] [color={rgb, 255:red, 155; green, 155; blue, 155 }  ,draw opacity=1 ][line width=0.75]    (10.93,-3.29) .. controls (6.95,-1.4) and (3.31,-0.3) .. (0,0) .. controls (3.31,0.3) and (6.95,1.4) .. (10.93,3.29)   ;
\draw [color={rgb, 255:red, 155; green, 155; blue, 155 }  ,draw opacity=1 ]   (221.25,155.89) -- (404.5,155.89) ;
\draw [shift={(406.5,155.89)}, rotate = 180] [color={rgb, 255:red, 155; green, 155; blue, 155 }  ,draw opacity=1 ][line width=0.75]    (10.93,-3.29) .. controls (6.95,-1.4) and (3.31,-0.3) .. (0,0) .. controls (3.31,0.3) and (6.95,1.4) .. (10.93,3.29)   ;
\draw    (249,40.39) -- (248.97,156.23) ;
\draw    (230.29,92.03) -- (382.53,173.4) ;
\draw [color={rgb, 255:red, 155; green, 155; blue, 155 }  ,draw opacity=1 ]   (486.37,184.14) -- (486.37,21.14) ;
\draw [shift={(486.37,19.14)}, rotate = 90] [color={rgb, 255:red, 155; green, 155; blue, 155 }  ,draw opacity=1 ][line width=0.75]    (10.93,-3.29) .. controls (6.95,-1.4) and (3.31,-0.3) .. (0,0) .. controls (3.31,0.3) and (6.95,1.4) .. (10.93,3.29)   ;
\draw [color={rgb, 255:red, 155; green, 155; blue, 155 }  ,draw opacity=1 ]   (421,155.89) -- (642,155.89) ;
\draw [shift={(644,155.89)}, rotate = 180] [color={rgb, 255:red, 155; green, 155; blue, 155 }  ,draw opacity=1 ][line width=0.75]    (10.93,-3.29) .. controls (6.95,-1.4) and (3.31,-0.3) .. (0,0) .. controls (3.31,0.3) and (6.95,1.4) .. (10.93,3.29)   ;
\draw    (486.47,155.73) .. controls (544.01,155.89) and (573.38,155.32) .. (621.92,130.46) ;
\draw    (428.5,47.39) -- (486.47,155.73) ;
\draw    (432.7,98.39) -- (623,168.04) ;
\draw  [dash pattern={on 0.84pt off 2.51pt}]  (28.35,179.02) -- (47.97,155.73) ;
\draw  [dash pattern={on 0.84pt off 2.51pt}]  (486.47,155.73) -- (501.35,183.42) ;
\draw  [dash pattern={on 0.84pt off 2.51pt}]  (248.95,178.02) -- (248.97,156.23) ;
\draw  [draw opacity=0][dash pattern={on 4.5pt off 4.5pt}] (65.78,137.75) .. controls (70.99,142.73) and (74.02,149.84) .. (73.48,157.51) .. controls (73.48,157.6) and (73.47,157.69) .. (73.46,157.77) -- (47.97,155.73) -- cycle ; \draw  [dash pattern={on 4.5pt off 4.5pt}] (65.78,137.75) .. controls (70.99,142.73) and (74.02,149.84) .. (73.48,157.51) .. controls (73.48,157.6) and (73.47,157.69) .. (73.46,157.77) ;  
\draw    (248.97,156.23) .. controls (306.51,156.39) and (335.88,155.82) .. (384.42,130.96) ;
\draw  [draw opacity=0][dash pattern={on 4.5pt off 4.5pt}] (451.95,90.03) .. controls (454.85,88.51) and (458.14,87.64) .. (461.64,87.64) .. controls (473.22,87.64) and (482.6,97.12) .. (482.6,108.81) .. controls (482.6,111.28) and (482.18,113.66) .. (481.41,115.86) -- (461.64,108.81) -- cycle ; \draw  [dash pattern={on 4.5pt off 4.5pt}] (451.95,90.03) .. controls (454.85,88.51) and (458.14,87.64) .. (461.64,87.64) .. controls (473.22,87.64) and (482.6,97.12) .. (482.6,108.81) .. controls (482.6,111.28) and (482.18,113.66) .. (481.41,115.86) ;  
\draw  [draw opacity=0][dash pattern={on 4.5pt off 4.5pt}] (97.14,101.1) .. controls (102.3,106.13) and (105.27,113.26) .. (104.67,120.93) .. controls (104.34,125.08) and (103,128.91) .. (100.91,132.21) -- (79.17,118.93) -- cycle ; \draw  [dash pattern={on 4.5pt off 4.5pt}] (97.14,101.1) .. controls (102.3,106.13) and (105.27,113.26) .. (104.67,120.93) .. controls (104.34,125.08) and (103,128.91) .. (100.91,132.21) ;  
\draw  [draw opacity=0][dash pattern={on 4.5pt off 4.5pt}] (249.51,131.12) .. controls (263.37,131.4) and (274.53,142.52) .. (274.54,156.2) .. controls (274.54,156.29) and (274.54,156.38) .. (274.54,156.47) -- (248.97,156.23) -- cycle ; \draw  [dash pattern={on 4.5pt off 4.5pt}] (249.51,131.12) .. controls (263.37,131.4) and (274.53,142.52) .. (274.54,156.2) .. controls (274.54,156.29) and (274.54,156.38) .. (274.54,156.47) ;  
\draw  [draw opacity=0][dash pattern={on 4.5pt off 4.5pt}] (249.41,76.53) .. controls (263.27,76.81) and (274.43,87.93) .. (274.44,101.61) .. controls (274.45,105.95) and (273.33,110.04) .. (271.36,113.6) -- (248.87,101.64) -- cycle ; \draw  [dash pattern={on 4.5pt off 4.5pt}] (249.41,76.53) .. controls (263.27,76.81) and (274.43,87.93) .. (274.44,101.61) .. controls (274.45,105.95) and (273.33,110.04) .. (271.36,113.6) ;  
\draw  [draw opacity=0][dash pattern={on 4.5pt off 4.5pt}] (477.05,136.82) .. controls (479.97,135.33) and (483.28,134.51) .. (486.77,134.56) .. controls (498.32,134.72) and (507.55,144.29) .. (507.43,155.94) -- (486.47,155.73) -- cycle ; \draw  [dash pattern={on 4.5pt off 4.5pt}] (477.05,136.82) .. controls (479.97,135.33) and (483.28,134.51) .. (486.77,134.56) .. controls (498.32,134.72) and (507.55,144.29) .. (507.43,155.94) ;  
\draw    (47.97,155.73) .. controls (105.51,155.89) and (134.88,155.32) .. (183.42,130.46) ;
\draw  [dash pattern={on 0.84pt off 2.51pt}]  (79.17,156.13) -- (79.17,118.93) ;
\draw  [dash pattern={on 0.84pt off 2.51pt}]  (135.8,154.84) -- (135.8,146.44) ;
\draw  [dash pattern={on 0.84pt off 2.51pt}]  (461.6,156.04) -- (461.64,108.81) ;
\draw  [dash pattern={on 0.84pt off 2.51pt}]  (572.6,156.04) -- (572.6,153.24) -- (572.6,148.44) ;
\draw  [dash pattern={on 0.84pt off 2.51pt}]  (337.4,156.04) -- (337.4,147.64) ;

\draw (188.2,118.7) node [anchor=north west][inner sep=0.75pt]  [font=\footnotesize]  {$u$};
\draw (388.6,119.1) node [anchor=north west][inner sep=0.75pt]  [font=\footnotesize]  {$u$};
\draw (625.8,117.9) node [anchor=north west][inner sep=0.75pt]  [font=\footnotesize]  {$u$};
\draw (153.6,19.2) node [anchor=north west][inner sep=0.75pt]  [font=\footnotesize]  {$\partial H$};
\draw (225.4,32.44) node [anchor=north west][inner sep=0.75pt]  [font=\footnotesize]  {$\partial H$};
\draw (413.2,32.44) node [anchor=north west][inner sep=0.75pt]  [font=\footnotesize]  {$\partial H$};
\draw (185.2,170.84) node [anchor=north west][inner sep=0.75pt]  [font=\footnotesize]  {$\ell $};
\draw (386.4,171.64) node [anchor=north west][inner sep=0.75pt]  [font=\footnotesize]  {$\ell $};
\draw (626.8,166.44) node [anchor=north west][inner sep=0.75pt]  [font=\footnotesize]  {$\ell $};
\draw (60,142.2) node [anchor=north west][inner sep=0.75pt]  [font=\scriptsize]  {$\gamma $};
\draw (256.4,140.6) node [anchor=north west][inner sep=0.75pt]  [font=\scriptsize]  {$\gamma $};
\draw (488.8,142.4) node [anchor=north west][inner sep=0.75pt]  [font=\scriptsize]  {$\gamma $};
\draw (90.4,109.8) node [anchor=north west][inner sep=0.75pt]  [font=\scriptsize]  {$\alpha $};
\draw (256,88.6) node [anchor=north west][inner sep=0.75pt]  [font=\scriptsize]  {$\alpha $};
\draw (463.2,95.4) node [anchor=north west][inner sep=0.75pt]  [font=\scriptsize]  {$\alpha $};
\draw (585.4,158.6) node [anchor=north west][inner sep=0.75pt]  [font=\scriptsize]  {$\overline{x}$};
\draw (58.4,125.4) node [anchor=north west][inner sep=0.75pt]  [font=\scriptsize]  {$r$};
\draw (238.4,118.6) node [anchor=north west][inner sep=0.75pt]  [font=\scriptsize]  {$r$};
\draw (465.8,132) node [anchor=north west][inner sep=0.75pt]  [font=\scriptsize]  {$r$};
\draw (344.6,158.6) node [anchor=north west][inner sep=0.75pt]  [font=\scriptsize]  {$\overline{x}$};
\draw (144.8,158.6) node [anchor=north west][inner sep=0.75pt]  [font=\scriptsize]  {$\overline{x}$};
\draw (75.2,157.8) node [anchor=north west][inner sep=0.75pt]  [font=\scriptsize]  {$x_{0}$};
\draw (456.8,158.6) node [anchor=north west][inner sep=0.75pt]  [font=\scriptsize]  {$x_{0}$};
\draw (568,157.8) node [anchor=north west][inner sep=0.75pt]  [font=\scriptsize]  {$x_{1}$};
\draw (331.2,157) node [anchor=north west][inner sep=0.75pt]  [font=\scriptsize]  {$x_{1}$};
\draw (131.2,157.4) node [anchor=north west][inner sep=0.75pt]  [font=\scriptsize]  {$x_{1}$};
\end{tikzpicture}

        \caption{Contact angle $\gamma \in (0, \pi/2)$, $\gamma = \pi/2$ and $\gamma \in (\pi/2, \pi)$.}
        \label{picYoungleq}
    \end{figure}
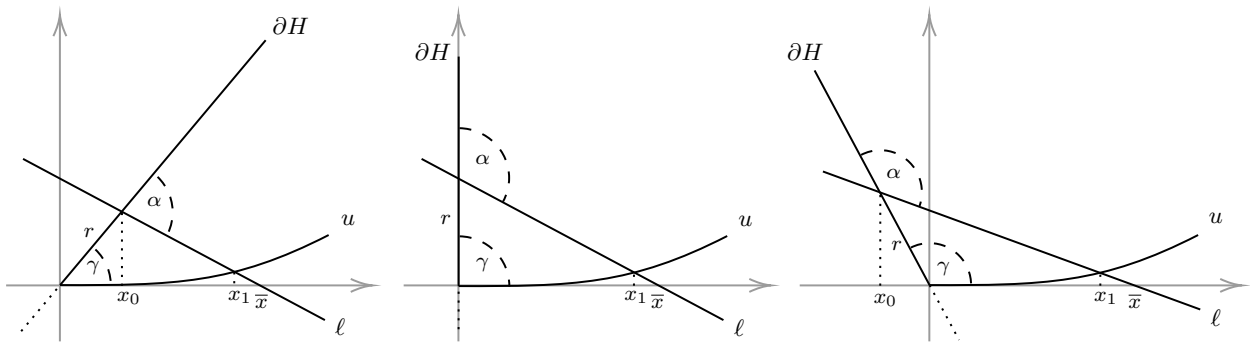
    By $\Lambda$-minimality of $E$ under convexity constrain we obtain
    \[  P_\beta(E)+ V(E)\le P_\beta(F)+ V(F) + \Lambda |E \setminus F|, \]
    which writes as
    \begin{equation}\label{innermin} P(E, H^\circ) - P(F, H^\circ) \leq \beta \lt( P(E, \partial H) - P(F, \partial H )\rt) + (V(F) - V(E)) + \Lambda |E \setminus F|. \end{equation}
    We now estimate each of the terms separately. We first notice that since $u'(x)=o(1)$ as $x \rightarrow 0$, a Taylor expansion and \eqref{point x_1} yield
    \[ \int_0^{x_1} \lt(1 + (u')^2 \rt)^{1/2} = x_1 + o(x_1) = \bar x + o_{\alpha}(\bar x) \qquad \text{as } \bar x \rightarrow 0.  \]
    Hence, recalling that $l'=\tan(\alpha - \gamma)$ and using the explicit formula \eqref{point x_0}, we can compute
    \[ \begin{split}
        P(E, H^\circ) - P(F, H^\circ) & = \int_0^{x_1} \lt(1 + (u')^2 \rt)^{1/2} - \int_{x_0}^{x_1} \lt(1 + (l')^2 \rt)^{1/2} \\ & = \bar x + o_{\alpha}(\bar x) - (\bar x + o_{\alpha}(\bar x) - x_0) \lt( 1+ \tan^2(\alpha - \gamma)\rt)^{1/2} \\ & = \bar x \lt(1- \lt(1- \frac{\sin(\alpha-\gamma)\cos \gamma}{\sin \alpha} \rt)\lt( 1+ \tan^2(\alpha - \gamma)\rt)^{1/2} \rt) + o_{\alpha}(\bar x).
    \end{split}\]
    By definition, we also have
    \[P(E, \partial H) - P(F, \partial H ) = r= \frac{\sin(\alpha-\gamma)}{\sin \alpha} \bar x. \]
    We estimate the volume difference $|E \setminus F|$. In the case $\gamma \in (0, \pi/2)$ we have
    \[ |E \setminus F|= \int_0^{x_0} (x\tan \gamma-u) + \int_{x_0}^{x_1} (l-u) \leq \int_0^{x_0} x\tan \gamma + \int_{x_0}^{\bar x} l = O_\alpha(\bar x^2) = o_\alpha(\bar x).   \]
    If instead $\gamma \in (\pi/2, \pi)$ we get the same result by computing
    \[ |E \setminus F| = \int_{x_0}^0 (l- x \tan \gamma) + \int_0^{x_1} (l - u) \leq \int_{x_0}^{\bar x} l = O_{\alpha}(\bar x^2)= o_{\alpha}(\bar x). \]
    It remains to treat the term $V(F) - V(E)$. We assume that $R$ is small enough so that for all $t < R$ (and hence for $\bar x$) we have $\partial E \cap B_t = \{y_1^t, y_2^t\}$. We take $t= \bar x$ and we also assume $\alpha$ close enough to $\gamma$ so that $E \setminus F \subset B_t$. In this way we obtain $E_{0,t}  \subset F$, where $E_{0, t} \in \cC_E$ denotes the set obtained from $E$ via the cutting procedure described above. Thus, Assumption \ref{assloclip} yields
    \[ V(F) - V(E) \leq V(E_{0, \bar x}) - V(E) \lesssim  t^{\min(q,s)} = o(\bar x).\]
    Plugging all the estimates back in \eqref{innermin} we find 
    \[ \bar x\lt( 1-\lt(1-\frac{\sin(\alpha-\gamma)\cos \gamma}{\sin \alpha}\rt)\lt( 1+ \tan^2(\alpha - \gamma)\rt)^{1/2} \rt) \leq \frac{\sin(\alpha-\gamma)}{\sin \alpha} \beta \bar x + o(\bar x) +o_\alpha(\bar x).\]
    Dividing by $\bar x$ and sending $\bar x \rightarrow 0$ we obtain
    \[ 1-\lt(1-\frac{\sin(\alpha-\gamma)\cos \gamma}{\sin \alpha}\rt)\lt( 1+ \tan^2(\alpha - \gamma)\rt)^{1/2} 
    \leq \frac{\sin(\alpha-\gamma)}{\sin \alpha}\beta. \]
    We now use a Taylor expansion at $\alpha = \gamma$ to infer
    \[ 1-\lt(1-(\alpha-\gamma)\frac{\cos \gamma}{\sin \gamma} + o(\alpha - \gamma)\rt)\lt( 1+ o(\alpha - \gamma)\rt) 
    \leq \frac{(\alpha-\gamma)}{\sin \gamma}\beta + o(\alpha - \gamma), \]
    from which we get
    \[ \frac{(\alpha-\gamma)}{\sin \gamma} \cos \gamma \leq \frac{(\alpha-\gamma)}{\sin \gamma} \beta + o(\alpha - \gamma). \]
    Dividing by $(\alpha - \gamma)/\sin \gamma$ (recall that $\sin \gamma >0$) and sending $\alpha \rightarrow \gamma$ we conclude $\cos \gamma \leq \beta$ when $\gamma \neq \pi/2$. The case $\gamma = \pi/2$ is similar and we only point out the main differences. Here we have $\partial H = \{(0, y): y \in \R \}$ and the set $E$ satisfies
    \[ E\cap B_R=\{ (x,y): x\geq 0, y \ge u(x)\} \cap B_R. \]
    The linear function $l$ is defined as before and we consider as a competitor the set $F \subset H$ such that $F \setminus B_R = E \setminus B_R$ and
    \[ F \cap B_R = \lt\{ (x,y):x \geq 0, y \geq \max(u(x), l(x))\rt\}\cap B_R. \]
    In this case there holds $r=\mathcal{H}^1(\partial H\cap (E\setminus F)) = l(0)= \tan(\alpha-\gamma)\bar x$. The four terms appearing in \eqref{innermin} are estimated as above and altogether we find
    \[ \begin{split}
        \bar x \lt(1- \lt( 1+ \tan^2(\alpha - \gamma)\rt)^{1/2} \rt) + o_{\alpha}(\bar x) \leq \tan(\alpha-\gamma)\beta \bar x + o_{\alpha}(\bar x).
    \end{split}\]
    Dividing by $\bar x$ and sending $\bar x \rightarrow 0$ we obtain $1- \lt( 1+ \tan^2(\alpha - \gamma)\rt)^{1/2} \leq \tan(\alpha-\gamma)\beta$ and another Taylor expansion yields $ o(\alpha - \gamma) \leq (\alpha - \gamma)\beta$. Dividing by $(\alpha - \gamma)$ and sending $\alpha \rightarrow \gamma$, we conclude $0=\cos \gamma \leq \beta$ also in this case. \end{proof}



\end{section}

\begin{section}{Regularity of charged liquid drops and Young's law in the planar case}\label{sec3}

In this section we apply the results of the previous one to minimizers of \eqref{minproblem}. Let $E \subset H$ be a minimizer of \eqref{minproblem}. By Remark \ref{non-empty contact}, $\partial E \cap \partial H\neq \emptyset$. Following the notation introduced at the beginning of Section \ref{sec5}, let $P_1,P_2 \in \partial E \cap \partial H$ be the contact points of $E$ (with possibly $P_1=P_2$) and $\gamma_1, \gamma_2$ be the associated contact angles.

\begin{theorem}\label{younglog1}
    Let $E$ be a minimizer of \eqref{minproblem} with contact angle $\gamma$, then $\partial E \cap \{x_2>0\}$ is $C^{1,1}_{loc}$ regular, $\cH^1(\partial E \cap \partial H)>0$ and $\cos \gamma \leq \beta$. Moreover, for $Q_0>0$ small enough, $Q \leq Q_0$ and $\delta>0$, the $C^{1,1}$ character of $\partial E \cap \{x_2>\delta\}$ depends only on $Q_0$, $\delta$ and $\beta$.
\end{theorem}

\begin{proof}
Fix $Q_0>0$. By Lemma \ref{volconstraint}, there exists $\Lambda=\Lambda(Q_0)$ such that $E$ is a minimizer of the unconstrained problem \eqref{minproblemfree} for every $Q\le Q_0$. Hence,   it is a $\Lambda$-minimizer of the capillarity shape functional
\[ J_\beta=\Pb + Q^2\cI_2\]
under convexity constraint, see Remark \ref{remminprop}. Moreover, as pointed out in Remark \ref{assumption ok}, $Q^2\cI_2$ satisfies Assumption \ref{assloclip}. We thus obtain $C^{1, \eta}_{loc}$ regularity of $\partial E \cap \{x_2>0\}$ by applying Theorem \ref{C1,etareg}. Let us now improve this to $C^{1,1}_{loc}$ regularity. Writing $\mu_E = f_E \llcorner \partial E$, since $u_E$ solves \eqref{logequilibrium},
by \cite[Corollary 8.36]{gilbarg1977elliptic} we have that $u_E$ is $C^{1,\eta}$ regular up to $\partial E \cap \{x_2 > \delta\}$ with $C^{1,\eta}$ norm depending on the $C^{1,\eta}$ character of $\partial E \cap \{x_2 > \delta\}$. Thanks to Lemma \ref{abscontmeas} we find
\[ f_E = \frac{|\nabla u_E|}{2\pi} \in C^{0,\eta}(\partial E \cap \{x_2 > \delta\}), \]
so that $f_E \in L^{\infty}(\partial E \cap \{x_2>\delta\})$ with $L^{\infty}$ norm depending on the quantities listed above. By \textit{Step 5} of the proof of \cite[Theorem 4.4]{GolNovRuf16} we find that  Assumption \ref{assloclip} holds with $q=3$ and $s=2$ so that $C^{1,1}_{loc}$ regularity follows. Using Remark \ref{uniform estimates}, Lemma \ref{hausdorffconv} and the fact that $B^\beta\cap\{x_2>\delta\}$ is smooth gives the uniformity of the estimates for small charges. Finally, by Theorem \ref{young1} we have  $\cH^1(\partial E \cap \partial H)>0$ and $\cos \gamma \leq \beta$. \end{proof}



\begin{remark}
    Since $\nabla u_E\notin L^\infty (\partial E\cap H)$, in light of the Euler--Lagrange condition \eqref{eulerlagintro} we do not expect that $\kappa_E \in L^{\infty}(\partial E \cap \{x_2>0\})$.
\end{remark}

We now show that if the charge $Q$ is small enough then the contact angle $\gamma$ satisfies Young's law $\cos \gamma = \beta$. We begin with a preparatory result that is a consequence of the minimality of $E$.

\begin{lemma}\label{criticalcharge}
    Let $E$ be a minimizer of \eqref{minproblem} and $F\subset H$ be a convex body such that $E \subset F$. Setting
    \[\delta \text{Vol}:= |F \setminus E| = |F| - 1 \qquad \text{and} \qquad \delta P := \Pb(F) - \Pb(E), \]
    there exist $\bar Q>0$ such that for every $Q \leq \bar Q$ we have $\delta \text{Vol} \lesssim \delta P$ whenever $\delta \text{Vol} \ll 1$.
\end{lemma}

\begin{proof}
    Let $\lambda  < 1$ be such that $|\lambda F|=1$. By Taylor expansion,
    \begin{equation}\label{taylorlam} \lambda = \lt( 1+ \delta \text{Vol} \rt)^{-1/2} = 1 - \frac{1}{2}\delta \text{Vol} + o(\delta \text{Vol}) \qquad \text{as } \delta \text{Vol} \rightarrow 0^+. \end{equation}
    Using $\lambda F$ as a competitor for \eqref{minproblem}, the minimality of $E$ yields
    \[ \begin{split}\Pb(E) + Q^2 \cI_2(E) & \leq \Pb(\lambda F) + Q^2 \cI_2(\lambda F) \\ &= \lambda \Pb(F) + Q^2 \cI_2(F) - Q^2 \log(\lambda).   \end{split} \]
    Using $\cI_2(F) \leq \cI_2 (E)$ and rearranging the terms we find
    \[ (1 - \lambda) \Pb(E) + Q^2 \log\lambda \leq \lambda \delta P \leq \delta P. \] 
    By hypothesis we have $1 - \lambda \simeq \delta \text{Vol} \ll 1$, so by means of the Taylor expansion \eqref{taylorlam} we obtain \[ \delta \text{Vol} \lt( \Pb(E) - Q^2 \rt) \lesssim \delta P. \]
    Recalling the convergence $\Pb(E) \longrightarrow \Pb(B^{\beta})$ as $Q \rightarrow 0$ given by \eqref{perconv},  we choose $Q$ small enough so that $\Pb(E) - Q^2 \gtrsim 1$ and conclude that indeed $\delta \text{Vol} \lesssim \delta P$.
\end{proof}

We now construct new competitors for  \eqref{minproblem}. In the upcoming proofs we will use them to show first the second inequality of Young's law and then an additional property of the minimizers of $E$.

\begin{construction}\label{buildcomp}
    Let $z \in \partial E \cap \{x_2>0\}$ and $e \in \S^1$ such that $z+\R e$ is the tangent line to $\partial E$ in $z$. There exist $R, r>0$ such that the following holds. Setting $K:=[-R, R] \times [-r, r]$, there exists a convex function of class $C^{1,1}$ with $u(0)=u'(0)=0$ for which we have
    \[ \partial E \cap K \cap H = \{z+xe+u(x)e^\perp:x \in \R\} \cap K \cap H. \]
    Thus, $E \cap K$ coincides with the epigraph of $u$ in $K \cap H$. The set $K$ is introduced in this form so that flat subsets of $\partial E \cap \{x_2>0\}$  can be entirely parametrized as a graph (see for example Theorem \ref{strictconv} below). From now on we use $\{e,e^\perp\}$ as a basis of $\R^2$ and $z$ as the origin, we keep denoting by $H$ the half-space appearing in \eqref{minproblem} expressed in these new coordinates. We introduce some additional notation, see Figure \ref{picbuildcomp}.
    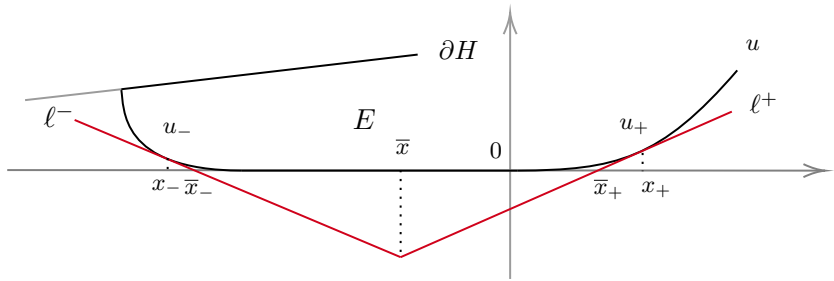
\begin{figure}
        \centering
\tikzset{every picture/.style={line width=0.75pt}} 

\begin{tikzpicture}[x=0.75pt,y=0.75pt,yscale=-1,xscale=1]

\draw [color={rgb, 255:red, 128; green, 128; blue, 128 }  ,draw opacity=1 ]   (70.5,95.89) -- (481,96.17) ;
\draw [shift={(483,96.18)}, rotate = 180.04] [color={rgb, 255:red, 128; green, 128; blue, 128 }  ,draw opacity=1 ][line width=0.75]    (10.93,-3.29) .. controls (6.95,-1.4) and (3.31,-0.3) .. (0,0) .. controls (3.31,0.3) and (6.95,1.4) .. (10.93,3.29)   ;
\draw [color={rgb, 255:red, 155; green, 155; blue, 155 }  ,draw opacity=1 ]   (324.14,150.84) -- (324.14,18.34) ;
\draw [shift={(324.14,16.34)}, rotate = 90] [color={rgb, 255:red, 155; green, 155; blue, 155 }  ,draw opacity=1 ][line width=0.75]    (10.93,-3.29) .. controls (6.95,-1.4) and (3.31,-0.3) .. (0,0) .. controls (3.31,0.3) and (6.95,1.4) .. (10.93,3.29)   ;
\draw    (188.8,96.18) -- (257.32,96.18) -- (327.89,96.18) ;
\draw    (327.89,96.18) .. controls (377.32,95.89) and (399.82,90.68) .. (438.68,45.56) ;
\draw    (128,55.09) .. controls (129.5,64.59) and (126.41,95.89) .. (188.8,96.18) ;
\draw [color={rgb, 255:red, 208; green, 2; blue, 27 }  ,draw opacity=1 ]   (104.5,70.59) -- (268.91,139.85) ;
\draw [color={rgb, 255:red, 208; green, 2; blue, 27 }  ,draw opacity=1 ]   (435.95,66.38) -- (268.91,139.85) ;
\draw  [dash pattern={on 0.84pt off 2.51pt}]  (390.95,87.21) -- (390.95,96.47) ;
\draw  [dash pattern={on 0.84pt off 2.51pt}]  (151.41,89.81) -- (151.41,96.75) ;
\draw  [dash pattern={on 0.84pt off 2.51pt}]  (268.91,95.89) -- (268.91,139.85) ;
\draw    (128,55.09) -- (277.5,37.59) ;
\draw [color={rgb, 255:red, 155; green, 155; blue, 155 }  ,draw opacity=1 ]   (79.5,60.59) -- (116.8,56.36) -- (128,55.09) ;

\draw (243.25,63.69) node [anchor=north west][inner sep=0.75pt]    {$E$};
\draw (443.55,55.11) node [anchor=north west][inner sep=0.75pt]  [font=\small]  {$\ell ^{+}$};
\draw (87.5,61.02) node [anchor=north west][inner sep=0.75pt]    {$\ell ^{-}$};
\draw (312.36,80.62) node [anchor=north west][inner sep=0.75pt]  [font=\footnotesize]  {$0$};
\draw (388.91,100.49) node [anchor=north west][inner sep=0.75pt]  [font=\footnotesize]  {$x_{+}$};
\draw (141.86,99.1) node [anchor=north west][inner sep=0.75pt]  [font=\footnotesize]  {$x_{-}$};
\draw (377.73,68.67) node [anchor=north west][inner sep=0.75pt]  [font=\footnotesize]  {$u_{+}$};
\draw (147.5,70.6) node [anchor=north west][inner sep=0.75pt]  [font=\footnotesize]  {$u_{-}$};
\draw (441.59,28.6) node [anchor=north west][inner sep=0.75pt]  [font=\footnotesize]  {$u$};
\draw (265,78.51) node [anchor=north west][inner sep=0.75pt]  [font=\footnotesize]  {$\overline{x}$};
\draw (158.41,98.79) node [anchor=north west][inner sep=0.75pt]  [font=\footnotesize]  {$\overline{x}_{-}$};
\draw (364.91,98.79) node [anchor=north west][inner sep=0.75pt]  [font=\footnotesize]  {$\overline{x}_{+}$};
\draw (286.75,29.19) node [anchor=north west][inner sep=0.75pt]  [font=\small]  {$\partial H$};

\end{tikzpicture}
        \caption{Construction \ref{buildcomp}}
        \label{picbuildcomp}
    \end{figure}
    Let $x_+, x_- \in (-R, R)$ with $x_- < x_+$. We set 
    \[ u_\pm:= u(x_\pm) \qquad \text{and} \qquad u'_\pm:= u'(x_\pm), \]
    so that $(x_+, u_+), (x_-, u_-)\in \text{Graph}(u)$. We call $\ell^+$ and $\ell^-$ the tangent lines to $\partial E$ passing through $(x_+, u_+)$ and $(x_-, u_-)$ respectively. In particular, we have
    \[\ell^\pm:=\{(x, l^\pm(x)):x \in \R \} \qquad \text{with} \qquad l^\pm(x):= u_\pm + u'_\pm(x - x_\pm).\]
    Whenever either $u'_+\neq 0$ or $u'_- \neq 0$, we also set
    \[ \bar x_+:= x_+ - \frac{u_+}{u'_+} \qquad \text{and} \qquad \bar x_-:= x_- - \frac{u_-}{u'_-}= x_- + \frac{u_-}{-u'_-},\]
    so that $(\bar x_+, 0)$ and $(\bar x_-, 0)$ are the intersections between the $x$ axis and the lines $\ell^+$ and $\ell^-$ respectively. If both $u'_+\neq 0$ and $u'_- \neq 0$ then we set $\bar x$ such that $l^+(\bar x)=l^-(\bar x)$. We take as a competitor the set $F \subset H$ such that
    \[ F:=E \cup \lt \{(x,y): x \in [x_-,x_+], \, \max (l^-(x), l^+(x)) \leq y \leq u(x)\rt \}.\]
    By construction $F$ is convex and $E \subset F$.
\end{construction}

We are ready to show the full validity of Young's law $\cos \gamma = \beta$ for the contact angle of our minimizer $E$. To this end, we prove that if the charge is small enough then the contact set $\partial E \cap \partial H$ is the only flat subset of $\partial E$.

\begin{theorem}\label{strictconv}
    Let $E$ be a minimizer of \eqref{minproblem} and let $\bar Q$ be given by Lemma \ref{criticalcharge}, then for every $Q \leq \bar Q$ the set $\partial E \cap \{x_2>0\}$ contains no segments.
\end{theorem}

\begin{proof}
    Assume by contradiction that there exists a segment $\Gamma \subset \partial E \cap \{x_2>0\}$ and set $2L:=\cH^1(\Gamma)$. Then there exist $z \in \partial E \cap \{x_2>0\}$ with tangent unit vector $e$ such that
    \[ \Gamma = [z-Le, z+Le]. \]
    We distinguish two cases.
    \begin{itemize}
        \item[1.] $\Gamma \cap \partial H = \emptyset$, namely no endpoint of $\Gamma$ is a contact point of $E$ with $\partial H$.
        \item[2.] $\Gamma \cap \partial H \neq \emptyset$, so an endpoint of $\Gamma$ is a contact point $P$ of $E$ with $\partial H$. Up to changing sign to $e$ we assume that $P= z - Le$.
    \end{itemize}
    From now on we use the coordinate system in which $z$ is the origin and $\{e, e^{\perp}\}$ is the orthonormal basis. Adopting the notation introduced in Construction \ref{buildcomp}, we set $r>0$, $R>L$ and we parametrize $\partial E \cap K \cap H$ as the graph of a convex and $C^{1,1}$ function $u: \R \rightarrow \R$ in $K \cap H$. We point out that by hypothesis $u$ is constantly equal to $0$ on $[-L, L]$. For $h \ll r$, we build a competitor for \eqref{minproblem} by applying Construction \ref{buildcomp} with $\bar x =(0, -h)$. Let $(x_+, u_+)$ and $(x_-, u_-)$ be the associated tangency points, see Figure \ref{picstrict} for an example of \textit{Case 1} and \textit{Case 2}.
    \begin{figure}
        \centering
\tikzset{every picture/.style={line width=0.75pt}} 

\begin{tikzpicture}[x=0.75pt,y=0.75pt,yscale=-1,xscale=1]

\draw [color={rgb, 255:red, 128; green, 128; blue, 128 }  ,draw opacity=1 ]   (25.17,96.34) -- (310.4,96.34) ;
\draw [shift={(312.4,96.34)}, rotate = 180] [color={rgb, 255:red, 128; green, 128; blue, 128 }  ,draw opacity=1 ][line width=0.75]    (10.93,-3.29) .. controls (6.95,-1.4) and (3.31,-0.3) .. (0,0) .. controls (3.31,0.3) and (6.95,1.4) .. (10.93,3.29)   ;
\draw [color={rgb, 255:red, 155; green, 155; blue, 155 }  ,draw opacity=1 ]   (168.3,167.09) -- (168.2,20.74) ;
\draw [shift={(168.2,18.74)}, rotate = 89.96] [color={rgb, 255:red, 155; green, 155; blue, 155 }  ,draw opacity=1 ][line width=0.75]    (10.93,-3.29) .. controls (6.95,-1.4) and (3.31,-0.3) .. (0,0) .. controls (3.31,0.3) and (6.95,1.4) .. (10.93,3.29)   ;
\draw    (84.47,96.47) -- (251.47,96.22) ;
\draw    (251.47,96.22) .. controls (276.3,96.34) and (300.3,85.84) .. (309.3,60.09) ;
\draw    (41.8,62.59) .. controls (44.28,74.59) and (50.81,96.28) .. (84.47,96.47) ;
\draw [color={rgb, 255:red, 208; green, 2; blue, 27 }  ,draw opacity=1 ]   (167.8,144.59) -- (284.8,88.84) ;
\draw [color={rgb, 255:red, 208; green, 2; blue, 27 }  ,draw opacity=1 ]   (167.8,144.59) -- (59.3,90.59) ;
\draw  [dash pattern={on 0.84pt off 2.51pt}]  (284.8,88.84) -- (284.8,96.84) ;
\draw  [dash pattern={on 0.84pt off 2.51pt}]  (59.3,90.59) -- (59.3,98.59) ;
\draw [color={rgb, 255:red, 128; green, 128; blue, 128 }  ,draw opacity=1 ]   (343.7,96.34) -- (628.93,96.34) ;
\draw [shift={(630.93,96.34)}, rotate = 180] [color={rgb, 255:red, 128; green, 128; blue, 128 }  ,draw opacity=1 ][line width=0.75]    (10.93,-3.29) .. controls (6.95,-1.4) and (3.31,-0.3) .. (0,0) .. controls (3.31,0.3) and (6.95,1.4) .. (10.93,3.29)   ;
\draw [color={rgb, 255:red, 155; green, 155; blue, 155 }  ,draw opacity=1 ]   (486.83,167.09) -- (486.73,20.74) ;
\draw [shift={(486.73,18.74)}, rotate = 89.96] [color={rgb, 255:red, 155; green, 155; blue, 155 }  ,draw opacity=1 ][line width=0.75]    (10.93,-3.29) .. controls (6.95,-1.4) and (3.31,-0.3) .. (0,0) .. controls (3.31,0.3) and (6.95,1.4) .. (10.93,3.29)   ;
\draw    (403,96.47) -- (570,96.22) ;
\draw    (570,96.22) .. controls (594.83,96.34) and (618.83,85.84) .. (627.83,60.09) ;
\draw [color={rgb, 255:red, 208; green, 2; blue, 27 }  ,draw opacity=1 ]   (486.33,144.59) -- (603.33,88.84) ;
\draw [color={rgb, 255:red, 208; green, 2; blue, 27 }  ,draw opacity=1 ]   (486.33,144.59) -- (403,96.47) ;
\draw  [dash pattern={on 0.84pt off 2.51pt}]  (603.33,88.84) -- (603.33,96.84) ;
\draw    (403,96.47) -- (415,40.01) ;
\draw  [dash pattern={on 4.5pt off 4.5pt}]  (403,96.47) -- (394.33,137.26) ;
\draw  [draw opacity=0][dash pattern={on 4.5pt off 4.5pt}] (410.2,72.38) .. controls (421.12,75.54) and (428.85,85.6) .. (428.56,97.14) -- (403,96.47) -- cycle ; \draw  [dash pattern={on 4.5pt off 4.5pt}] (410.2,72.38) .. controls (421.12,75.54) and (428.85,85.6) .. (428.56,97.14) ;  

\draw (139.55,54.74) node [anchor=north west][inner sep=0.75pt]    {$E$};
\draw (304.7,41.34) node [anchor=north west][inner sep=0.75pt]    {$u$};
\draw (247.23,84.71) node [anchor=north west][inner sep=0.75pt]  [font=\footnotesize]   {$L$};
\draw (75.23,84.21) node [anchor=north west][inner sep=0.75pt]  [font=\footnotesize]   {$-L$};
\draw (53.8,98.91) node [anchor=north west][inner sep=0.75pt]  [font=\footnotesize]   {$x_{-}$};
\draw (279.87,99.24) node [anchor=north west][inner sep=0.75pt]  [font=\footnotesize]   {$x_{+}$};
\draw (170.13,109.58) node [anchor=north west][inner sep=0.75pt]  [font=\footnotesize]   {$h$};
\draw (222.47,118.88) node [anchor=north west][inner sep=0.75pt]  [font=\small]  {$\ell^{+}$};
\draw (103.13,119.21) node [anchor=north west][inner sep=0.75pt]  [font=\small]  {$\ell^{-}$};
\draw (458.08,54.74) node [anchor=north west][inner sep=0.75pt]    {$E$};
\draw (623.23,41.34) node [anchor=north west][inner sep=0.75pt]    {$u$};
\draw (565.77,84.71) node [anchor=north west][inner sep=0.75pt]  [font=\footnotesize]   {$L$};
\draw (382.77,85.21) node [anchor=north west][inner sep=0.75pt]  [font=\footnotesize]   {$-L$};
\draw (598.4,99.24) node [anchor=north west][inner sep=0.75pt]  [font=\footnotesize]   {$x_{+}$};
\draw (488.67,109.91) node [anchor=north west][inner sep=0.75pt]  [font=\footnotesize]   {$h$};
\draw (541,118.88) node [anchor=north west][inner sep=0.75pt]  [font=\small]  {$\ell^{+}$};
\draw (433.33,119.88) node [anchor=north west][inner sep=0.75pt]  [font=\small]  {$\ell^{-}$};
\draw (393.67,34.4) node [anchor=north west][inner sep=0.75pt]  [font=\footnotesize]   {$\partial H$};
\draw (409.4,80.6) node [anchor=north west][inner sep=0.75pt]  [font=\footnotesize]   {$\gamma $};

    \end{tikzpicture}
    \caption{\textit{Cases 1} and \textit{Case 2} with contact angle $\gamma < \pi/2$.}
    \label{picstrict}
    \end{figure}
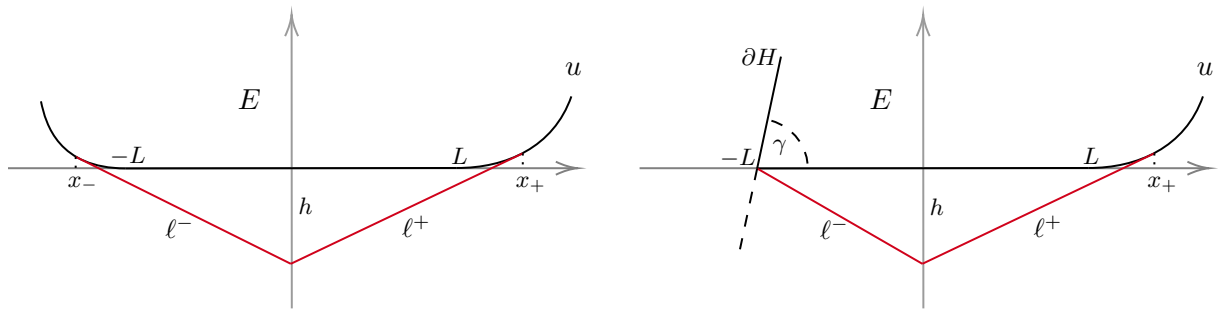
     If $h$ is small enough then $|x_+ - L|\ll 1$ and $|x_-+L| \ll1$. Moreover, in \textit{Case 2} we have $x_-=-L$ for $h$ small enough.  Since $Q \leq \bar Q$, by Lemma \ref{criticalcharge} we have $\delta \text{Vol} \lesssim \delta P$ whenever $h$ is small enough. By construction the triangle with vertices $(L,0)$, $(-L,0)$ and $(0,-h)$ is a subset of $F \setminus E$, thus we have
    \begin{equation}\label{estdiffvol}
        \delta \text{Vol} \geq Lh.
    \end{equation}
    We now estimate $\delta P$ from above, writing $\delta P = \delta P_+ + \delta P_-$ with
    \[ \begin{split}
        \delta P_+ & := \int_0^{x_+} \lt( 1+ (u'_+)^2 \rt)^{\frac{1}{2}} - \int_0^{x_+} \lt( 1+ (u')^2 \rt)^{\frac{1}{2}} = \int_0^{x_+} \frac{(u'_+ + u')(u'_+ - u')}{\lt( 1+ (u'_+)^2 \rt)^{1/2} + \lt( 1+ (u'_+)^2 \rt)^{1/2}} \\ & \leq u'_+ \int_0^{x_+} (u'_+ - u' ) = u'_+ \int_0^{x_+} \int_{x}^{x_+} u''= \int_0^{x_+} tu''(t)dt,
    \end{split} \]
    where we used Fubini in the last equality. The term $\delta P_-$ can be treated in a similar way. Altogether, we find
    \[ \delta P \leq u'_+\int_0^{x_+} tu''(t)dt + |u'_-|\int_{x_-}^0 |t|u''(t)dt. \]
    Since $l^+(L)= l^+(0)+ u'_+L \le 0$ we have that $u'_+ \leq h/L$, in a similar way we get $|u'_-|\leq h/L$. Moreover, $x_+ \lesssim L$ and $|x_-| \lesssim L$ for $h$ small enough. Thus we obtain
    \begin{equation}\label{estdiffper}
        \delta P \leq (u'_+)^2 x_+ + (u'_-)^2|x_-| \lesssim \frac{h^2}{L}.
    \end{equation}
    Plugging \eqref{estdiffvol} and \eqref{estdiffper} in the relation $\delta \text{Vol} \lesssim \delta P$ we find $L^2 \lesssim h$, which gives a contradiction for $h$ small enough.
\end{proof}

\begin{remark}
    As a consequence of Theorem \ref{younglog1} and Theorem \ref{strictconv}, set $\partial E \cap \{x_2>0\}$ can be locally represented by graphs of strictly convex functions of class $C^{1,1}$. Therefore, the set $\{x \in \partial E \cap H:\kappa_E(x)=0\}$ has empty interior in the relative topology of $\partial E \cap \{x_2>0\}$.
\end{remark}

\begin{proof}[Proof of Theorem \ref{thm:Youngintro}]
   By Theorem \ref{younglog1}, we just need  to show that $\cos \gamma \geq \beta$. Let $\bar Q$ be given by Lemma \ref{criticalcharge}. In turn this follows combining  Theorem \ref{strictconv} and Theorem \ref{young2}.
\end{proof}

We conclude the paper with the following result which follows arguing similarly to Theorem \ref{strictconv} .

\begin{lemma}\label{positivecurvature}
    Let $E$ be a minimizer of \eqref{minproblem} and let $\bar Q$ be given by Lemma \ref{criticalcharge}. There exists $\eps_0>0$ such that, for every $Q \leq \bar Q$, if $z \in \partial E \cap \{x_2>0\}$ is a Lebesgue point of $\kappa_E$ with $\kappa_E(z) >0$ then we have $\kappa_E(z) \geq \eps_0$.
\end{lemma}

\begin{proof}
    We refer to the notation introduced in Construction \ref{buildcomp}. There exists $r>0$ small enough such that in a suitable system of coordinates where $z$ is the origin, we can parametrize $\partial E\cap B_r(z)$ as the graph of a strictly convex and $C^{1,1}$ function $u: \R \rightarrow \R$ for which $u(0)=u'(0)=0$. By hypothesis, $0$ is a Lebesgue point of $u''$. In particular
    \begin{equation}\label{secondderivative}
    a:=u''(0)=\lim_{x \rightarrow 0} \frac{1}{2x} \int_{-x}^x u''>0. \end{equation}
    Since $u(0)=u'(0)=0$, we have the relations
    \begin{equation}\label{tayrloru''}
        u(x)= \frac{a}{2}x^2+o(x^2) \qquad \text{and} \qquad u'(x)= ax + o(x) \quad \text{as } x \rightarrow 0.
    \end{equation} 
    Let $0<x \ll 1$, we  apply Construction \ref{buildcomp} with $x_+=x$ and $x_-=-x$. In particular, the point $\bar x$ such that $l^+(\bar x)= l^-(\bar x)$ satisfies
    \[ \bar x = x\frac{u'_+ + u'_-}{u'_+ - u'_-}- \frac{u_+ - u_-}{u'_+ - u'_-}.\]
    Thanks to \eqref{tayrloru''} and the fact that $a>0$ we have $u'_+-u'_-=O(x)$, $u'_++u'_-=o(x)$ and $u_+-u_-=o(x^2)$, thus 
    \begin{equation}\label{keyvanish} \lim_{x \rightarrow 0} \frac{\bar x}{x}=0.\end{equation}
     Since $Q \leq \bar Q$, by Lemma \ref{criticalcharge} we have $\delta \text{Vol} \lesssim \delta P$ whenever $x$ is small enough. We estimate $\delta P$ from above by arguing as in the proof of Theorem \ref{strictconv}. We write $\delta P = \delta P_+ + \delta P_-$ with
     \[ \delta P_+:=\int_{\bar x}^{x} \lt( 1+ (u'_+)^2 \rt)^{\frac{1}{2}} - \int_{\bar x}^{x}  \lt( 1+ (u')^2 \rt)^{\frac{1}{2}} \leq u'_+ \int_{\bar x}^{x}  (t - \bar x)u''(t)\,dt. \]
     The term $\delta P_-$ is defined and estimated in a similar way and altogether we find
     \[ \delta P \leq u'_+ \int_{\bar x}^{x}  |t - \bar x|u''(t)\,dt + (-u'_-) \int_{-x}^{\bar x} |t - \bar x|u''(t)\, dt \lesssim \lt( \int_{-x}^x u'' \rt)^2 x. \]
     We now compute $\delta \text{Vol}$. As before, we write $\delta \text{Vol} = \delta \text{Vol}_+ + \delta \text{Vol}_-$ with 
     \[ \begin{split}
         \delta \text{Vol}_+ & := \int_{\bar x}^x (u - l^+) = \int_{\bar x}^x ( u(t)-u_+ - u'_+(t - x))\,dt = \int_{\bar x}^x \int_{x}^{t} (u'(y) - u'_+) \,dy\,dt \\ & = \int_{\bar x}^x \int_t^x \int_y^x u''(s)\,ds\,dy\,dt = \frac{1}{2} \int_{\bar x}^x (s-\bar x)^2 u''(s)\,ds, 
     \end{split}\]
    where we used Fubini to get the last equality. A similar estimate holds for $\delta \text{Vol}_- = \delta \text{Vol} - \delta \text{Vol}_+$. Altogether, the relation $\delta \text{Vol} \lesssim \delta P$ yields
    \begin{equation}\label{poscurv1} \int_{-x}^x(s - \bar x)^2 u''(s)\,ds \lesssim \lt( \int_{-x}^x u'' \rt)^2 x. \end{equation}
    By \eqref{keyvanish} let $x$ be small enough such that $\bar x \in (-x/4, x/4)$, then $(s - \bar x)^2 \gtrsim x^2$ for all $s \in (x/2, x)$ and we have
    \begin{equation}\label{poscurv2} 
    \int_{-x}^x(s - \bar x)^2 u''(s)\,ds \gtrsim x^2 \lt( \int_{-x}^{-x/2} u'' + \int_{x/2}^x u'' \rt) \gtrsim x^2 \int_{-x}^x u''. \end{equation}
    The last equality is another consequence of the fact that $a>0$, indeed under this condition there holds
    \[ \lim_{x \rightarrow 0} \bigg( \int_{-x/2}^{x/2}u''\bigg) \lt( \int_{-x}^{x}u'' \rt)^{-1}= \lim_{x \rightarrow 0} \bigg( \frac{1}{x} \int_{-x/2}^{x/2}u''\bigg) \lt( \frac{1}{x}\int_{-x}^{x}u'' \rt)^{-1} = \frac{1}{2}, \]
    whence for $x$ small enough
    \[ \int_{x/2}^x u'' \gtrsim \int_0^x u'' \qquad \text{and} \qquad \int_{-x}^{-x/2} u'' \gtrsim \int_{-x}^0 u''.\]
    Thanks to \eqref{poscurv1} and \eqref{poscurv2} we obtain
    \[ 1 \lesssim \frac{1}{x} \int_{-x}^x u''. \]
    Taking the limit as $x \rightarrow 0$ we conclude that as claimed, $a=u''(0)\gtrsim 1$. \end{proof}

\end{section}

{\bf Acknowledgements.} This work was supported by Fondation Mathématique Jacques Hadamard, by the project G24-202 funded by Università Italo Francese, by the ANR Stoiques and by Next Generation EU, PRIN 2022E9CF89 and PRIN PNRR P2022WJW9H.

\bibliographystyle{acm}
\bibliography{main}

\begin{thebibliography}{10}

\bibitem{bellettini2002total}
{\sc Bellettini, G., Caselles, V., and Novaga, M.}
\newblock The total variation flow in {$\mathbb{R}^n$}.
\newblock {\em Journal of Differential Equations 184}, 2 (2002), 475--525.

\bibitem{carazzato2025quantitative}
{\sc Carazzato, D., Pascale, G., and Pozzetta, M.}
\newblock Quantitative isoperimetric inequalities in capillarity problems and cones in strong and barycentric forms.
\newblock {\em arXiv preprint arXiv:2507.07686\/} (2025).

\bibitem{caselles2007uniqueness}
{\sc Caselles, V., Chambolle, A., and Novaga, M.}
\newblock Uniqueness of the {C}heeger set of a convex body.
\newblock {\em Pacific Journal of Mathematics 232}, 1 (2007), 77--90.

\bibitem{chodosh2024improved}
{\sc Chodosh, O., Edelen, N., and Li, C.}
\newblock Improved regularity for minimizing capillary hypersurfaces.
\newblock {\em arXiv preprint arXiv:2401.08028\/} (2024).

\bibitem{de2025regularity}
{\sc De~Masi, L., Edelen, N., Gasparetto, C., and Li, C.}
\newblock Regularity of minimal surfaces with capillary boundary conditions.
\newblock {\em Communications on Pure and Applied Mathematics 78}, 12 (2025), 2436--2502.

\bibitem{DePHirVes19}
{\sc De~Philippis, G., Hirsch, J., and Vescovo, G.}
\newblock Regularity of minimizers for a model of charged droplets.
\newblock {\em Comm. Math. Phys.\/} (2019), 1--46.

\bibitem{dephilippis2015regularity}
{\sc De~Philippis, G., and Maggi, F.}
\newblock Regularity of free boundaries in anisotropic capillarity problems and the validity of young’s law.
\newblock {\em Archive for Rational Mechanics and Analysis 216}, 2 (2015), 473--568.

\bibitem{de2025rigidity}
{\sc De~Rosa, A., Neumayer, R., and Resende, R.}
\newblock Rigidity of critical points of hydrophobic capillary functionals.
\newblock {\em arXiv preprint arXiv:2509.22532\/} (2025).

\bibitem{esposito2011remark}
{\sc Esposito, L., and Fusco, N.}
\newblock A remark on a free interface problem with volume constraint.
\newblock {\em J. Convex Anal 18}, 2 (2011), 417--426.

\bibitem{esposito2005quantitative}
{\sc Esposito, L., Fusco, N., and Trombetti, C.}
\newblock A quantitative version of the isoperimetric inequality: the anisotropic case.
\newblock {\em Annali della Scuola Normale Superiore di Pisa-Classe di Scienze 4}, 4 (2005), 619--651.

\bibitem{finn1986equilibrium}
{\sc Finn, R.}
\newblock {\em Equilibrium Capillary Surfaces}, vol.~Die Grundlehren der mathematischen Wissenschaften, Vol. 284.
\newblock Springer-Verlag New York Inc., 1986.

\bibitem{fontelos2009variational}
{\sc Fontelos, M.~A., and Kindel{\'a}n, U.}
\newblock A variational approach to contact angle saturation and contact line instability in static electrowetting.
\newblock {\em The Quarterly Journal of Mechanics \& Applied Mathematics 62}, 4 (2009), 465--480.

\bibitem{fusco2025isoperimetric}
{\sc Fusco, N., Julin, V., Morini, M., and Pratelli, A.}
\newblock The isoperimetric inequality for the capillary energy outside convex sets.
\newblock {\em arXiv preprint arXiv:2509.10200\/} (2025).

\bibitem{garnett2005harmonic}
{\sc Garnett, J.~B., and Marshall, D.~E.}
\newblock {\em Harmonic measure}.
\newblock No.~2. Cambridge University Press, 2005.

\bibitem{gilbarg1977elliptic}
{\sc Gilbarg, D., and Trudinger, N.~S.}
\newblock {\em Elliptic partial differential equations of second order}, vol.~224.
\newblock Springer, 1977.

\bibitem{goldman2012volume}
{\sc Goldman, M., and Novaga, M.}
\newblock Volume-constrained minimizers for the prescribed curvature problem in periodic media.
\newblock {\em Calculus of Variations and Partial Differential Equations 44\/} (2012), 297--318.

\bibitem{GolNovRuf13}
{\sc Goldman, M., Novaga, M., and Ruffini, B.}
\newblock Existence and stability for a non-local isoperimetric model of charged liquid drops.
\newblock {\em Arch. Rat. Mech. Anal. 217}, 1 (2015), 1--36.

\bibitem{GolNovRuf16}
{\sc Goldman, M., Novaga, M., and Ruffini, B.}
\newblock On minimizers of an isoperimetric problem with long-range interactions under a convexity constraint.
\newblock {\em Anal. PDE 11}, 5 (2018), 1113–1142.

\bibitem{goldman2024charged}
{\sc Goldman, M., Novaga, M., and Ruffini, B.}
\newblock A charged liquid drop model with {W}illmore energy.
\newblock {\em arXiv preprint arXiv:2409.01045\/} (2024).

\bibitem{Gol22}
{\sc Goldman, M., Novaga, M., and Ruffini, B.}
\newblock Rigidity of the ball for an isoperimetric problem with strong capacitary repulsion.
\newblock {\em Journal of the European Mathematical Society\/} (2024).

\bibitem{kreutz2024note}
{\sc Kreutz, L., and Schmidt, B.}
\newblock A note on the {W}interbottom shape.
\newblock {\em Proceedings of the Royal Society of Edinburgh Section A: Mathematics\/} (2024), 1--14.

\bibitem{lamboley2023regularity}
{\sc Lamboley, J., and Prunier, R.}
\newblock Regularity in shape optimization under convexity constraint.
\newblock {\em Calculus of Variations and Partial Differential Equations 62}, 3 (2023), 101.

\bibitem{doohovskoy_foundations_1972}
{\sc Landkof, N.}
\newblock {\em Foundations of {Modern} {Potential} {Theory}}.
\newblock Grundlehren der mathematischen {Wissenschaften}. Springer Berlin Heidelberg, 1972.

\bibitem{lieb2001analysis}
{\sc Lieb, E.~H., and Loss, M.}
\newblock {\em Analysis}, vol.~14.
\newblock American Mathematical Soc., 2001.

\bibitem{lippmann1875relations}
{\sc Lippmann, G.}
\newblock Relations entre les ph{\'e}nom{\`e}nes {\'e}lectriques et capillaires.
\newblock {\em Ann. Chim. Phys 5}, 11 (1875), 494--549.

\bibitem{maggi2012sets}
{\sc Maggi, F.}
\newblock {\em Sets of finite perimeter and geometric variational problems: an introduction to Geometric Measure Theory}.
\newblock No.~135. Cambridge University Press, 2012.

\bibitem{mugele2005electrowetting}
{\sc Mugele, F., and Baret, J.-C.}
\newblock Electrowetting: from basics to applications.
\newblock {\em Journal of physics: condensed matter 17}, 28 (2005), R705.

\bibitem{mugele2019electrowetting}
{\sc Mugele, F., and Heikenfeld, J.}
\newblock {\em Electrowetting: fundamental principles and practical applications}.
\newblock John Wiley \& Sons, 2019.

\bibitem{muratov2016well}
{\sc Muratov, C.~B., and Novaga, M.}
\newblock On well-posedness of variational models of charged drops.
\newblock {\em Proceedings of the Royal Society A: Mathematical, Physical and Engineering Sciences 472}, 2187 (2016), 20150808.

\bibitem{muratov2018equilibrium}
{\sc Muratov, C.~B., Novaga, M., and Ruffini, B.}
\newblock On equilibrium shape of charged flat drops.
\newblock {\em Communications on Pure and Applied Mathematics 71}, 6 (2018), 1049--1073.

\bibitem{pacati2025some}
{\sc Pacati, A., Tortone, G., and Velichkov, B.}
\newblock Some remarks on singular capillary cones with free boundary.
\newblock {\em arXiv preprint arXiv:2502.07697\/} (2025).

\bibitem{pascale2025existence}
{\sc Pascale, G.}
\newblock Existence and nonexistence of minimizers for classical capillarity problems in presence of nonlocal repulsion and gravity.
\newblock {\em Nonlinear Analysis 251\/} (2025), 113685.

\bibitem{pascale2024quantitative}
{\sc Pascale, G., and Pozzetta, M.}
\newblock Quantitative isoperimetric inequalities for classical capillarity problems.
\newblock {\em Calculus of Variations and Partial Differential Equations 63}, 9 (2024), 225.

\bibitem{prats2023notes}
{\sc Prats, M., and Tolsa, X.}
\newblock Notes on harmonic measure, 2025.

\bibitem{prunier2024fuglede}
{\sc Prunier, R.}
\newblock Fuglede-type arguments for isoperimetric problems and applications to stability among convex shapes.
\newblock {\em SIAM Journal on Mathematical Analysis 56}, 2 (2024), 1560--1603.

\bibitem{saff2013logarithmic}
{\sc Saff, E.~B., and Totik, V.}
\newblock {\em Logarithmic potentials with external fields}, vol.~316.
\newblock Springer Science \& Business Media, 2013.

\bibitem{scheid2009proof}
{\sc Scheid, C., and Witomski, P.}
\newblock A proof of the invariance of the contact angle in electrowetting.
\newblock {\em Mathematical and computer modelling 49}, 3-4 (2009), 647--665.

\bibitem{schneider2013convex}
{\sc Schneider, R.}
\newblock {\em Convex bodies: the Brunn--Minkowski theory}, vol.~151.
\newblock Cambridge university press, 2013.

\bibitem{valdinoci2025toward}
{\sc Valdinoci, E.}
\newblock Toward a long-range theory of capillarity.
\newblock {\em Notices of the American Mathematical Society 72}, 2 (2025).

\bibitem{young1805iii}
{\sc Young, T.}
\newblock An essay on the cohesion of fluids.
\newblock {\em Philosophical transactions of the royal society of London}, 95 (1805), 65--87.

\end{thebibliography}

\end{document}